\documentclass[final]{dmtcs-episciences}

\usepackage[utf8]{inputenc}
\usepackage{subfigure}
\hypersetup{
    colorlinks=true,       % Enable colored links
    linkcolor=blue,        % Color of internal links (sections, references)
    citecolor=blue,        % Color of citation links
    filecolor=magenta,     % Color of file links
    urlcolor=black          % Color of external URLs
}
\usepackage{multirow}
\usepackage{amsthm}
\usepackage{float}
\newtheorem{theorem}{Theorem}

\newtheorem{prop}[theorem]{Proposition}
\newtheorem{obs}[theorem]{Observation}

\usepackage{amsmath}
\usepackage{arydshln}
\usepackage[most]{tcolorbox}
\newcommand{\ph}{\prec_P}
\newcommand{\pl}{\prec_L}

\newcommand{\nid}{\noindent}

\newcommand{\R}{\mathcal{R}}

\usepackage{tikz}
\tikzstyle{wall}=[line width=0.6pt]

\newcommand{\td}
{%
\texorpdfstring
{\begin{tikzpicture}[baseline=0.0mm,scale=0.10]
\protect\draw[wall] (0,1.7)--(2,1.7);
\protect\draw[wall] (1,0)--(1,1.7);
\end{tikzpicture}}%
{<tr>}%
}
\newcommand{\tu}
{%
\texorpdfstring
{\begin{tikzpicture}[baseline=0.0mm,scale=0.10]
\protect\draw[wall] (0,0)--(2,0);
\protect\draw[wall] (1,0)--(1,1.7);
\end{tikzpicture}}%
{<tr>}%
}
\newcommand{\tl}
{%
\texorpdfstring
{\begin{tikzpicture}[baseline=0.0mm,scale=0.10]
\protect\draw[wall] (1.7,0)--(1.7,2);
\protect\draw[wall] (0,1)--(1.7,1);
\end{tikzpicture}}%
{<tr>}%
}
\newcommand{\tr}
{%
\texorpdfstring
{\begin{tikzpicture}[baseline=0.0mm,scale=0.10]
\protect\draw[wall] (0,0)--(0,2);
\protect\draw[wall] (0,1)--(1.7,1);
\end{tikzpicture}}%
{<tr>}%
}

\newcommand{\tdt}
{%
\texorpdfstring
{\begin{tikzpicture}[baseline=0.0mm,scale=0.18]
\protect\draw[wall] (0,1.7)--(2,1.7);
\protect\draw[wall] (1,0)--(1,1.7);
\end{tikzpicture}}%
{<tr>}%
}
\newcommand{\tut}
{%
\texorpdfstring
{\begin{tikzpicture}[baseline=0.0mm,scale=0.18]
\protect\draw[wall] (0,0)--(2,0);
\protect\draw[wall] (1,0)--(1,1.7);
\end{tikzpicture}}%
{<tr>}%
}
\newcommand{\tlt}
{%
\texorpdfstring
{\begin{tikzpicture}[baseline=0.0mm,scale=0.18]
\protect\draw[wall] (1.7,0)--(1.7,2);
\protect\draw[wall] (0,1)--(1.7,1);
\end{tikzpicture}}%
{<tr>}%
}
\newcommand{\trt}
{%
\texorpdfstring
{\begin{tikzpicture}[baseline=0.0mm,scale=0.18]
\protect\draw[wall] (0,0)--(0,2);
\protect\draw[wall] (0,1)--(1.7,1);
\end{tikzpicture}}%
{<tr>}%
}

\newcommand{\crs}{
\begin{tikzpicture}[baseline=0.1mm,scale=0.12]
\protect\draw[wall] (0,1)--(2,1);
\protect\draw[wall] (1,0)--(1,2);
\end{tikzpicture}
}

\newcommand{\wma}{
\begin{tikzpicture}[baseline=0.6mm,scale=0.1]
\protect\draw[wall] (1.1,0)--(1.1,2.5);
\protect\draw[wall] (3.6,1.1)--(1.1,1.1);
\protect\draw[wall] (2.5,3.6)--(2.5,1.1);
\protect\draw[wall] (0,2.5)--(2.5,2.5);
\end{tikzpicture}
}

\newcommand{\wmb}{
\begin{tikzpicture}[baseline=0.6mm,scale=0.1]
\protect\draw[wall] (2.5,0)--(2.5,2.5);
\protect\draw[wall] (3.6,2.5)--(1.1,2.5);
\protect\draw[wall] (1.1,3.6)--(1.1,1.1);
\protect\draw[wall] (0,1.1)--(2.5,1.1);
\end{tikzpicture}
}

\newcommand{\zwall}{
\begin{tikzpicture}[baseline=0.6mm,scale=0.1]
\protect\draw[wall] (1.8,0)--(1.8,3.6);
\protect\draw[wall] (0.4,2.5)--(1.8,2.5);
\protect\draw[wall] (1.8,1.1)--(3.2,1.1);
\end{tikzpicture}
}

\newcommand{\swall}{
\begin{tikzpicture}[baseline=0.6mm,scale=0.1]
\protect\draw[wall] (0,1.8)--(3.6,1.8);
\protect\draw[wall] (2.5,1.8)--(2.5,3.2);
\protect\draw[wall] (1.1,0.4)--(1.1,1.8);
\end{tikzpicture}
}

\newcommand{\izwall}{
\begin{tikzpicture}[baseline=0.6mm,scale=0.1]
\protect\draw[wall] (1.8,0)--(1.8,3.6);
\protect\draw[wall] (1.8,2.5)--(3.2,2.5);
\protect\draw[wall] (0.4,1.1)--(1.8,1.1);
\end{tikzpicture}
}

\newcommand{\iswall}{
\begin{tikzpicture}[baseline=0.6mm,scale=0.1]
\protect\draw[wall] (0,1.8)--(3.6,1.8);
\protect\draw[wall] (2.5,0.4)--(2.5,1.8);
\protect\draw[wall] (1.1,1.8)--(1.1,3.2);
\end{tikzpicture}
}

\author[Andrei Asinowski and Michaela A. Polley]{Andrei Asinowski\affiliationmark{1}\thanks{The research of Andrei Asinowski was funded by the Austrian Science Fund (FWF) [\href{https://doi.org/10.55776/P32731}{10.55776/P32731}].}
  \and Michaela A. Polley\affiliationmark{1,2}\thanks{The~research of Michaela A.~Polley was funded by a Fulbright--Austrian Marshall Plan Foundation Award for Research in Science and Technology. \\ For open access purposes, the authors have applied a CC BY public copyright license to any author-accepted manuscript version arising from this submission.}}
\title[Patterns in rectangulations: $\td$-like patterns]{Patterns in rectangulations. Part I: \\
$\tdt$-like patterns, inversion sequence classes $I(010, 101, 120, 201)$  
and $I(011, 201)$, and rushed Dyck paths}
\affiliation{
  University of Klagenfurt, Klagenfurt, Austria\\
  Dartmouth College, Hanover, NH, USA}
\begin{document}
% This is only used if you are compiling for a volume before vol 25
% \publicationdetails{VOL}{2015}{ISS}{NUM}{SUBM}
% This is the new form of collecting the data, starting with vol 25
\publicationdata{vol. 27:1, Permutation Patterns 2024}{2025}{17}{10.46298/dmtcs.15118}{2025-01-22; 2025-01-22; 2025-09-03}{2025-09-04}
\maketitle
\begin{abstract}
We initiate a systematic study of \textit{pattern avoidance in rectangulations}.
We give a formal definition of such patterns and investigate rectangulations that avoid $\td$-like patterns ---
the pattern $\td$ and its rotations. For~every $L \subseteq \{\td, \, \tl, \,  \tu,  \, \tr \}$ 
we enumerate $L$-avoiding rectangulations, both weak and strong.
In particular, we show $\td$-avoiding \textit{weak} rectangulations are enumerated by Catalan numbers
and construct bijections to several Catalan structures.
Then,~we prove that $\td$-avoiding \textit{strong} rectangulations are in bijection
with several classes of inversion sequences, among them $I(010,101,120,201)$ and $I(011,201)$ ---
which leads to a~solution of the conjecture that these classes are Wilf-equivalent.
Finally, we show that $\{\td, \tu\}$-avoiding strong rectangulations are in bijection
with recently introduced rushed Dyck paths.
\end{abstract}

%SEC 1 INTRO
\section{Introduction} 
A rectangulation $\mathcal{R}$ is a decomposition 
of a rectangle $R$ into a finite number of rectangles. 
Rectangulations appear naturally in integrated circuit design~\cite{LaPotinDirector1986, Lengauer2012}, 
combinatorial and geometric algorithms~\cite{Felsner2024,  KozmaSaranurak2016,Richter2022,WimerKorenCederbaum1988}, 
scientific data visualization~\cite{BuchinEppsteinLoefflerNoellenburgSilveira2016, KreveldSpeckmann2007}, 
architecture~\cite{Flemming1978,MitchellSteadmanLiggett1976,Steadman1983}, 
and also in visual art and design. 
They are of interest to combinatorialists due to their rich structure
and many links to other combinatorial objects. 
For~example, classes of rectangulations have been shown to be in bijection with 
permutation classes, binary trees, posets, Hopf algebras, planar maps,
among others~\cite{AckermanBarequetPinter2006bij,AsinowskiCardinalFelsnerFusy2024,
Felsner2013,FelsnerFusyNoyOrden2011,LawReading2012,Meehan2019bax,Meehan2019hopf,Reading2012}. 

Our work concerns pattern avoidance in rectangulations.
Rectangulation patterns were considered, explicitly or implicitly, in several earlier contributions
in the areas of combinatorics, computational geometry, and geometric algorithms.
A~rectangulation is \textit{guillotine} if and only if it avoids both \textit{windmills} $\wmb$ and $\wma$:
this early result was mentioned in many papers on algorithmic floorplaning;
to our knowledge, the first combinatorial paper where it appeared is~\cite{AckermanBarequetPinter2006bij} by Ackerman, Barequet, and Pinter.
Cardinal, Sacrist\'an, and Silveira~\cite{CardinalSacristanSilveira2018} showed that
a rectangulation is (strongly equivalent to) a diagonal one
if and only if it avoids the patterns $\izwall$ and $\swall$.
Generating functions for some families of pattern-avoiding rectangulations 
were found by Asinowski and Mansour~\cite{AsinowskiMansour2010}. 
Eppstein, Mumford, Speckmann, and Verbeek~\cite{EppsteinMumfordSpeckmannVerbeek2012} proved that a rectangulation 
is \textit{area-universal} is and only if every one of its segment is \textit{one-sided}: 
from the perspective of patterns, these are precisely 
$\{\izwall, \zwall, \swall, \iswall\}$-avoiding rectangulations.
More recently, pattern avoidance in rectangulations was considered by Merino and M\"utze~\cite{MerinoMuetze2023} in the context of algorithmic generation of combinatorial structures, where eight specific patterns were considered for demonstrating the developed algorithm.
Asinowski and Banderier~\cite{AsinowskiBanderier2024} provided an analytic solution for several models involving these patterns --- in particular, all models from~\cite{MerinoMuetze2023} that deal with weak guillotine rectangulations.
Finally, Asinowski, Cardinal, Felsner, and Fusy~\cite{AsinowskiCardinalFelsnerFusy2024} proved several results about correspondence 
of rectangulation patterns to permutation patterns --- in particular a characterization 
(conjectured by Merino and Mütze) of strong guillotine rectangulations by \textit{mesh patterns}.

Our goal is to initiate a~systematic treatment of this topic;
therefore, we start this contribution with giving a precise definition of pattern avoidance in rectangulations.
Then we explore the case of $\td$-like patterns --- that~is, the patterns
$\td$, $\tr$, $\tu$, and~$\tl$. 
For every possible subset $L$ of these patterns, we enumerate rectangulations that avoid~$L$, 
and we provide bijections with classes of structures, 
such as inversion sequences, permutations, and Dyck paths. 
Our main results include bijections between $\td$-avoiding strong rectangulations
and several classes of inversion sequences, among them $I(010,101,120,201)$ and $I(011,201)$.
At~the time of writing, the former of these classes is \textit{known} to be enumerated by the sequence \href{https://oeis.org/A279555}{A279555} 
of the \href{https://oeis.org/}{OEIS}~\cite{oeis}, 
and the latter is \textit{conjectured} to be enumerated by the same 
sequence~\cite{CallanMansour2023, YanLin2020}.
Our results include a proof of this conjecture, and, at the same time, 
offer the first combinatorial interpretation of \href{https://oeis.org/A279555}{A279555} other than as a class of inversion sequences.
Our further results include a bijection between $\{\td, \tu\}$-avoiding strong rectangulations
and a recently introduced family of rushed Dyck paths, enumerated by~\href{https://oeis.org/A287709}{A287709}. 
In summary, already the simplest cases of pattern avoidance in rectangulations
lead to interesting results and connections.

In the follow-up articles we study representation of rectangulation patterns by permutation patterns,
and explore the cases where avoided patterns consist of three segments. 
The ultimate goal of our study is to develop general approaches, similar
to the areas of pattern avoidance in permutations, posets, lattice paths, etc.
For instance, given a set of patterns $L$, 
we wish to know whether $L$-avoiding rectangulations are bijective to a~class of permutations or inversion sequences,
and whether the generating function is rational / algebraic / D-finite or not.
In earlier contributions such questions were handled ad hoc, we wish to develop a systematic approach.

The motivation behind investigating the $\td$-like patterns is two-fold.
These are the simplest patterns, and it is very natural to start the systematic study with exploring them
in order to see which phenomena can be observed when they determine a class of rectangulations
(for instance, the difference between the ``weak'' and the ``strong'' case).
On the other hand, since all the joints in generic rectangulations have a $\td$~shape,
any pattern can be regarded as composed of several $\td$-like patterns.
Hence, good understanding of their avoidance is an essential step towards 
a general theory.

\newpage

The paper is organized as follows:
In Section~\ref{sec:def} we provide the background and recall main definiti\-ons and basic results on rectangulations and inversion sequences.
In Section~\ref{sec:patt_avoid_rect} we give a precise \mbox{definition} of patterns in rectangulations.
Then we present our results, in particular:
in Section~\ref{sec:td_weak} we \mbox{provide} bijections between $\td$-avoiding 
weak rectangulations and several Catalan structures (Theorem~\ref{thm:td_weak});
in Sections \ref{sec:td_strong1} and~\ref{sec:td_strong2}, bijections 
between $\td$-avoiding strong rectangulations and avoidance classes of inversion sequences,
including $I(010, 101, 120, 201)$ and $I(011, 201)$
(Theorems~\ref{thm:strong1} and ~\ref{thm:strong2});
and, in Section~\ref{sec:td_tu_strong}, bijections 
between $\{\td, \tu\}$-avoiding strong rectangulations 
and rushed Dyck paths (Theorem~\ref{thm:rush}).

% SEC 2
\section{Definitions and basics}
\label{sec:def}
\subsection{Rectangulations}
\label{sec:def_rect}

In this expository section we recall definitions and basic results on rectangulations.
All the results are 
taken from~\cite{AckermanBarequetPinter2006bij, 
AckermanBarequetPinter2006num, AsinowskiCardinalFelsnerFusy2024,
CardinalSacristanSilveira2018, 
MerinoMuetze2023}.
In most cases we follow the conventions and the notation from~\cite{AsinowskiCardinalFelsnerFusy2024}. 

\paragraph{Rectangulations and segments.} \label{par:rectangulations_strong}
A rectangulation $\mathcal{R}$ of an axis-aligned rectangle $R$ is a decomposition of $R$ into finitely many interior-disjoint rectangles. 
The \textit{size} of $\mathcal{R}$ is the number of rectangles in the decomposition. 
See Figure~\ref{fig:equivalence} for several rectangulations of size $9$.

A \textit{segment} in $\mathcal{R}$ is a union of one or more sides of rectangles that form a straight line, 
is maximal in regards to this property, and is not one of the sides of $R$. 
We will always assume that the rectangulations are \textit{generic}, 
which means that segments do not cross each other; 
that is, two segments can meet in a joint of the shape 
$\td$, $\tu$, $\tr$, or $\tl$, but never $\crs$. 
Under this assumption, every rectangulation of size $n$ has precisely $n-1$ segments.
Given a segment $s$, 
a \textit{neighboring segment} of $s$ is a perpendicular segment one of whose endpoints lies on $s$,
and a~\textit{neighboring rectangle} of $s$ is a rectangle of $\R$ one of whose sides is included in~$s$.
Depending on the side from which a neighboring segment (or rectangle) meets $s$,
we call it a left, right, top, or bottom neighboring segment (or rectangle).
A segment is one-sided if it has neighboring segments on at most one of its sides.

\begin{figure}[htbp]
    \centering
    \includegraphics[width=\textwidth]{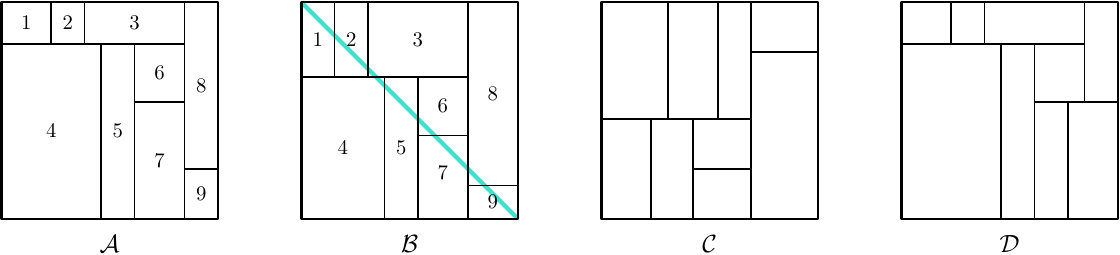}
    \caption{Four rectangulations of size $9$. 
    Rectangulations $\mathcal{A}$, $\mathcal{B}$, and $\mathcal{C}$ are weakly equivalent.
     Rectangulations $\mathcal{A}$ and~$\mathcal{B}$ are also strongly equivalent. Rectangulation $\mathcal{B}$ is diagonal. For rectangulations $\mathcal{A}$ and $\mathcal{B}$, the \textrm{NW}--\textrm{SE} labeling is shown.}
    \label{fig:equivalence}
\end{figure}

\paragraph{Weak and strong equivalence}\label{par:weak_strong}

There are two natural ways to define equivalence classes for rectangulations: 
the~\textit{weak equivalence} that preserves segment-to-rectangle contacts,
and the (finer) \textit{strong equivalence} that preserves rectangle-to-rectangle contacts.

A precise definition is based on \textit{neighborhood relations} between rectangles of $\R$.
A~rectangle~$Y$ is \textit{right} of a~rectangle $X$ (equivalently, $X$ is \textit{left} of $Y$)
if there exists a~sequence of rectangles $X=X_1, X_2, \ldots, X_k=Y$
such~that for every $i$, $1 \leq i \leq k-1$, there is a vertical segment $s_i$ 
such~that $X_i$ is a left neighbor of $s_i$ and $X_{i+1}$ is a~right neighbor of $s_i$.
A~rectangle~$Y$ is \textit{above} a rectangle $X$ (equivalently, $X$ is \textit{below} $Y$)
if there exists a~sequence of rectangles $X=X_1, X_2, \ldots, X_k=Y$
such that for every $i$, $1 \leq i \leq k-1$, there is a horizontal segment~$s_i$ 
such that~$X_i$ is a bottom neighbor of $s_i$ and $X_{i+1}$ is a top neighbor of $s_i$.
Every pair $X, \ Y$ of distinct rectangles of~$\R$ satisfies precisely one of these relations:
$X$ is either above, below, right of, or left of $Y$. 

Two rectangulations are \textit{weakly equivalent} if there is a (unique) bijection between their rectangles 
that preserves the left--right and above--below neighborhood relations.
Two rectangulations are \textit{strongly equivalent} if this bijection additionally preserves 
contacts between rectangles. 
Every weak rectangulation produces one or several strong rectangulations 
which differ from each other by \textit{shufflings} ---
the orders in which each segment meets the endpoints of all neighboring segments.

A \textit{weak rectangulation} is an equivalence class of rectangulations with respect to the weak equivalence, 
and a~\textit{strong rectangulation} is an equivalence class of rectangulations with respect to the strong equivalence. 
In~Figure~\ref{fig:equivalence}, the rectangulations $\mathcal{A}$, $\mathcal{B}$, and $\mathcal{C}$ are weakly equivalent, 
but only $\mathcal{A}$ and $\mathcal{B}$ are strongly equivalent. 
Hence, the four images in Figure~\ref{fig:equivalence} represent two distinct weak rectangulations but three distinct strong rectangulations.

\paragraph{Diagonal labelings}
We denote the sides of $R$ by the cardinal directions
$\textrm{N}$ (top), $\textrm{W}$ (left), $\textrm{S}$ (bottom), and $\textrm{E}$~(right), 
and, accordingly, the corners by $\textrm{NW}$, $\textrm{SW}$, $\textrm{SE}$, and $\textrm{NE}$.
A \textit{$\textrm{W}$-rectangle} is any rectangle that touches~$\textrm{W}$,
the \textit{$\textrm{NW}$-rectangle} is the rectangle that contains the $\textrm{NW}$-corner,
and similar for other sides and corners.
The~\textit{\textrm{NW}--\textrm{SE} ordering} of the rectangles of $\R$ is the order in which
$X \prec Y$ if and only if $X$ is left of or above~$Y$.
The \textit{\textrm{NW}--\textrm{SE} labeling} is the labeling of the rectangles by $1, 2, 3, \ldots, n$ 
according to the \textrm{NW}--\textrm{SE} ordering, 
see rectangulations $\mathcal{A}$ and $\mathcal{B}$ in Figure~\ref{fig:equivalence}.
In the \textrm{NW}--\textrm{SE} labeling, rectangle $Y$ is the direct successor of rectangle~$X$
if and only if there is a segment whose 
one endpoint is the \textrm{SE}-corner of~$X$ 
and the other endpoint is the \textrm{NW}-corner of $Y$.
One similarly defines \textrm{SW}--\textrm{NE}, \textrm{SE}--\textrm{NW}, and \textrm{NE}--\textrm{SW} orderings and labelings.

\paragraph{Diagonal rectangulations}
A \emph{diagonal rectangulation} is a rectangulation 
in which the \textrm{NW}--\textrm{SE} diagonal of~$R$ intersects the interior of every rectangle. 
In Figure~\ref{fig:equivalence}, 
$\mathcal{B}$ is a diagonal rectangulation.
Every rectangulation~$\mathcal{R}$ is weakly equivalent to a diagonal rectangulation. 
Moreover, all diagonal rectangulations weakly equivalent to~$\mathcal{R}$
lie in the same strong equivalence class, in which all the segments are shuffled so that
the configurations $\izwall$ and $\swall$ are avoided.
In a diagonal rectangulation, the \textrm{NW}--\textrm{SE} diagonal intersects the rectangles 
according to the \textrm{NW}--\textrm{SE} ordering.

\paragraph{Guillotine rectangulations}\label{par:guillotine}

A \textit{cut} in $\mathcal{R}$ is a segment whose endpoints lie on opposite sides of $R$. 
A rectangulation~$\mathcal{R}$ is \emph{guillotine} if it is either of size $1$, 
or contains a cut that splits it into two sub-rectangulations which are (recursively) both guillotine. 
In Figure~\ref{fig:equivalence}, the rectangulations $\mathcal{A}, \mathcal{B}, \mathcal{C}$ are guillotine, and $\mathcal{D}$ is non-guillotine.
A rectangulation is guillotine if and only if it avoids 
\textit{windmills} --- quadruples of segments of the shape
$\wma$ or $\wmb$.

\subsection{Inversion sequences}
\label{sec:def_invseq}

In this section we recall the basic definitions related to inversion sequences
and pattern avoidance in them.

An \emph{inversion sequence} of length $n$ is an integer sequence $e=(e_1,e_2,e_3,\dots,e_n)$ such that 
$0 \leq e_j \leq j-1$ holds for all $j\in [n]$. 
We denote the set of all inversion sequences of length $n$ by $I_n$. 
An inversion sequence of length $n$ can be plotted on the 
lower staircase of a $n\times n$ grid by labeling the columns from $1$~to~$n$ and the rows from $0$ to $n-1$, 
see Figure~\ref{fig:inv_seq} for an example. 

\begin{figure}[h]
	\centering
\includegraphics[scale=0.8]{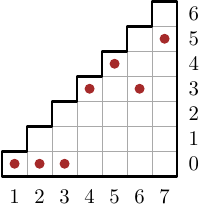}
	\caption{Plot of the inversion sequence $e=(0,0,0,3,4,3,5)$.}
	\label{fig:inv_seq}
\end{figure}

Given an inversion sequence $e$, we say that $e_j$ is \textit{high} if $e_j=j-1$. 
Further, the \textit{bounce}\footnote{This term was coined by Jay Pantone in~\cite{Pantone2024}.} of $e$ is 
$k:=n-M$, where $n$ is the length of $e$, and $M$ is its maximum value.
The inversion sequence $e=(0,0,0,3,4,3,5)$ from Figure~\ref{fig:inv_seq} has 
three high elements, three $0$ elements, and its bounce is $2$.
Its~left-to-right maxima are $0,3,4,5$, and its right-to-left minima are $0,3,5$.

There is a natural size-preserving bijection $\Theta$ from permutations to inversion sequences:
A permutation $\pi = \pi_1 \pi_2 \ldots \pi_n \in S_n$
is mapped to the inversion sequence $e=(e_1,e_2,e_3,\dots,e_n) \in I_n$,
where $e_k$ is the number of inversions whose second (smaller) element is~$\pi_k$,
that is, $e_k=|\{i\colon \, 1\leq i<k, \, \pi_i>\pi_k\}|$.
In~this case we say that~$e$ is the inversion sequence of $\pi$.
For example, $(0,0,0,3,4,3,5) \in I_7$ is the inversion sequence of $5673142 \in S_7$.

Similar to pattern avoidance in permutations 
(refer to~\cite{Bevan2015} for main definitions and basic results), 
one defines pattern avoidance in inversion sequences,
with the caveat that a pattern is not necessarily a valid inversion sequence itself,
but can be any sequence $f$ of non-negative integers with set of values $\{0, 1, 2, \ldots, \ell \}$
for some $\ell$.
For example, $(1,0)$ is not an inversion sequence, but it can be considered as a pattern.
An inversion sequence~$e$ \emph{contains} $f$ if $e$ has a~(not necessarily contiguous) subsequence whose entries have the same relative order as~$f$. Otherwise, $e$ \emph{avoids}~$f$. 
All the patterns in our paper are of the length at most $3$,
and we write them without parentheses and commas, for example $201$ rather then $(2, 0, 1)$.
For example, the inversion sequence from Figure~\ref{fig:inv_seq}
contains --- among others --- the patterns $001$, $010$, $021$, $102$, 
but avoids the patterns $100$, $101$, $120$,  $210$. 
Given a set of patterns $T$, an inversion sequence is \textit{$T$-avoiding}
if it avoids all the patterns from $T$.
We denote the set of $T$-avoiding inversion sequences by~$I(T)$, 
and the set of $T$-avoiding inversion sequences of length~$n$ by $I_n(T)$.
For two sets of patterns, $T$ and $T'$, 
we say that the classes $I(T)$ and $I(T')$ are \textit{Wilf-equivalent} 
if $|I_n(T)|=| I_n(T')|$ for every $n \geq 1$;
we denote this by $I(T) \sim I(T')$.
A systematic investigation of pattern avoidance in inversion sequences was started 
in~\cite{CorteelMartinezSavageWeselcouch2016,MansourShattuck2015, MartinezSavage2018},
and it is still a very active area of research.

%SEC 3
\section{Pattern avoidance in rectangulations}
\label{sec:patt_avoid_rect}

As mentioned above, we aim to developing a systematic study of patterns in rectangulations.
Therefore, in this section we give a formal definition of rectangulation patterns in terms of configurations of segments. 

\medskip

\noindent\textbf{Definition.} A \textit{configuration of segments} is a connected\footnote{
Considering disconnected patterns would inevitably involve configurations of both segments and rectangles
and lead to more involved definitions.
Yet in~\cite{AsinowskiBanderier2024, MerinoMuetze2023} some disconnected patterns were considered, rather \textit{ad hoc}.} set of horizontal and vertical segments 
such that two distinct segments never cross and never share an endpoint.
Two configurations of segments, $C$ and~$C'$, are 
\textit{weakly equivalent} if there is a (unique) bijection 
$\gamma\colon  s \mapsto s'$ between their segments such that 
\begin{enumerate}
\item Segments $s \in C$ and $s' \in C'$ have the same orientation (both are horizontal or both are vertical);
\item Segments $s$ and $t$ meet in $C$ if and only if segments $s'$ and $t'$ meet in $C'$, 
and the respective joints are of the same kind;
\item For every segment $s$, the order in which its neighboring segments meet it is preserved by $\gamma$, 
independently, for both sides of $s$. 
(More~precisely,  
if $s$~is vertical, and $t_1, t_2, \ldots, t_k$ are its left neighboring segments 
listed in the order in which $s$ meets them from the top to the bottom,
then $t'_1, t'_2, \ldots, t'_k$ are precisely the left neighboring segments of $s'$ listed in the order in which $s'$ meets them from the top to the bottom; similarly for the right neighboring segments of $s$; and similarly for a horizontal~$s$.)
\end{enumerate}
Finally, $C$ and $C'$ are \textit{strongly equivalent} if 
they are weakly equivalent and 
the bijection $\gamma$ also satisfies
\begin{enumerate}
\setcounter{enumi}{3}
\item For every segment $s$, the order in which \textit{all} the neighbors of $s$ meet it, is preserved by $\gamma$.
(More precisely, 
if $s$~is vertical, and $t_1, t_2, \ldots, t_k$ are precisely the neighboring segments of $s$ listed in the order in which $s$~meets them from the top to the bottom,
then $t'_1, t'_2, \ldots, t'_k$ are precisely the neighboring segments of $s'$ listed in the order in which $s'$ meets them from the top to the bottom; 
similarly for a horizontal $s$.)
\end{enumerate}

\medskip

\noindent\textbf{Definition.} A \textit{weak rectangulation pattern} is an equivalence class of configurations of segments with respect to the weak equivalence.
A \textit{strong rectangulation pattern} is an equivalence class of configuration of segments with respect to the strong equivalence.

A rectangulation $\mathcal{R}$ \textit{contains} a pattern $p$ if there is an injection 
from the segments of (any representative of) $p$ into
the segments of (any representative of) $\mathcal{R}$ which preserves incidences, orientations, and order of incidence of neighbors of the segments of $p$.
Otherwise, $\mathcal{R}$ \textit{avoids} $p$.

\medskip

Here, the order of incidence should be understood according to the kind of equivalence. 
For weak rectangulations we consider weak patterns, 
and for strong rectangulations we consider 
strong patterns\footnote{Avoidance of strong patterns in weak rectangulations is in general not well defined, and avoidance of weak patterns in strong rectangulations is equivalent to avoidance of several strong patterns.}.

The next figures illustrate these concepts and present some phenomena. 
In Figure~\ref{fig:ex_ws}, all three segment configurations are weakly equivalent, but only (a) and (b) are strongly equivalent.
Accordingly, (a) and~(b) represent the same strong pattern, and (c) a different strong pattern;
but all three configurations represent the same weak pattern.
The colors are used in order to make the correspondence visually clear.

\begin{figure}[h]
\centering
\includegraphics[scale=0.95]{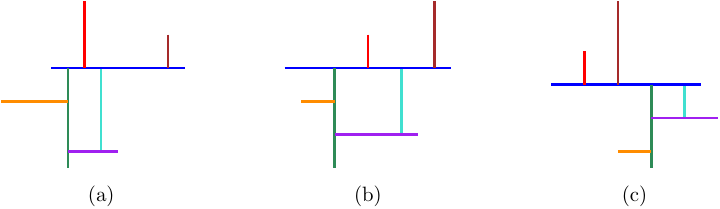} 
\caption{Images (a) and (b) represent the same strong pattern, and (c) a different strong pattern.
All three images represent the same weak pattern.}
\label{fig:ex_ws}
\end{figure}

Since patterns involve only segments, 
there is a greater flexibility in drawing their representatives than for rectangulations.
For example, images (a) and (b) in Figure~\ref{fig:ex} represent the same pattern.
(In this case, all the segments are one-sided, 
therefore this pattern can be seen as a weak or as a strong pattern.)
Moreover, an~occurrence of a pattern $p$ in a rectangulation preserves 
all the contacts that exist in $p$, but can also contain
some additional contacts: this is also illustrated in Figure~\ref{fig:ex}(c).

\begin{figure}[h]
\centering
\includegraphics[scale=0.95]{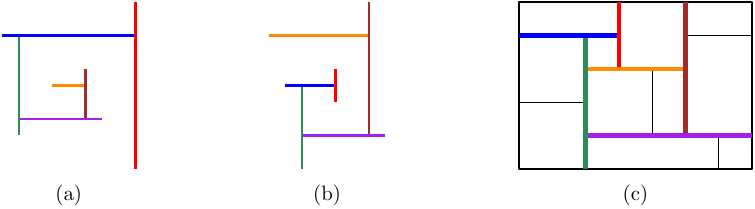} 
\caption{Two drawings of the same pattern (a,b), and its occurrence in a rectangulation (c).}
\label{fig:ex}
\end{figure}

Let $L$ be a set of rectangulation patterns.
Denote by $R^w(L)$ and $R^s(L)$ the set of weak and, respectively, strong rectangulations  that avoid all the patterns from $L$. We refer to these sets as \emph{rectangulation classes}.
Note~that if $L$ contains at least one pattern from $\{\td,\tu\}$ and at least one pattern from $\{\tr,\tl\}$,
then all the segments in the rectangulations of the class are one-sided,
and hence, the distinction between $R^w(L)$ and $R^s(L)$ is not essential:
in this case we simply write $R(L)$.
To refer to $L$-avoiding rectangulations of a fixed size $n$, 
we add a subscript and denote the respective sets by $R^w_n(L)$, $R^s_n(L)$, and $R_n(L)$.

% SEC 4
\section{$L$-avoiding rectangulations for $L \subseteq \{\tdt, \trt, \tut, \tlt \}$ }\label{sec:main}
In this section, we investigate the basic case of \textit{$\td$-like patterns} --- that is, the pattern $\td$ and its rotations $\tl, \tu, \tr$.
A summary of our main enumerative results is given in Table~\ref{tab:results}. 
There are five essentially different combinations~$L$~of $\td$-like patterns.
In Section~\ref{sec:td} we explore $\td$-avoiding rectangulations, 
and in Section~\ref{sec:td_tu}, $\{\td,\tu\}-$avoiding rectangulations.
Only in these two cases is distinction between  
weak and strong rectangulations (indicated by \textbf{w} and \textbf{s} in the table) essential. 
The last three cases are elementary, we give them in Section~\ref{sec:elementary}.

\renewcommand{\arraystretch}{1.3}
\begin{table}[h]
\begin{center}
\begin{tabular}{|c|c|c|c|}
\hline
$L$ & \textbf{w}/\textbf{s} & Formula/OEIS & Bijection to\ldots  \\ \hline 
\multirow{2}{3em}{\ $\{\td\}$} &\textbf{w}& \href{https://oeis.org/A000108}{A000108} (Catalan numbers)&
\begin{tabular}{c}
Binary trees, Dyck paths, \\  $(10)$-avoiding inversion sequences, \\ $(213)$-avoiding permutations
\end{tabular}
\\ \cline{2-4}
&\textbf{s}&\href{https://oeis.org/A279555}{A279555}& 
\begin{tabular}{c}
$I(010, 101, 120, 201)$, $I(010,110,120,210)$  \\
$I(010,100,120,210)$, and $I(011, 201)$.
\end{tabular}
\\ \hline 
\multirow{2}{3em}{$\{\td, \tu \}$} &\textbf{w}&$2^{n-1}$&Compositions  \\ \cline{2-4} &\textbf{s}&\href{https://oeis.org/A287709}{A287709}& Rushed Dyck paths  \\ \hline 
 \multicolumn{2}{|l|}{$\{\td, \tr \}$} &$2^{n-1}$& Binary sequences  \\ \hline 
\multicolumn{2}{|l|}{$\{\td, \tr, \tu \}$} &$n$&  \\ \hline 
\multicolumn{2}{|l|}{$\{\td, \tr, \tu, \tl \}$} &$2$&  \\ \hline 
\end{tabular}
\end{center}
\caption{Summary of main results on enumeration of $L$-avoiding rectangulations for  
$L \subseteq \{\td, \tr, \tu, \tl \}$.}
\label{tab:results}
\end{table}
\renewcommand{\arraystretch}{1}

\subsection{Enumeration and bijections for $R(\td)$}
\label{sec:td}
We start with the most basic case: $L = \{\td\}$. 
Note that $\td$-avoiding rectangulations are precisely those in which every vertical segment reaches $\mathrm{N}$.
We explore $\td$-avoiding weak rectangulations in Section~\ref{sec:td_weak}, and then $\td$-avoiding strong rectangulations in Sections~\ref{sec:td_strong1}--\ref{sec:td_strong2}.
In Section~\ref{sec:555} we provide a background on
$I(010,101,120,201)$ and $I(011,201)$ (and some other related classes of inversion sequences),
and in Section~\ref{sec:conj} we use our results to prove Wilf-equivalence of these two classes,
along with several matching statistics.

\subsubsection{$\td$-avoiding weak rectangulations: bijections with Catalan structures}
\label{sec:td_weak}
Enumeration of $R^w(\td)$ leads to a famous sequence:
\begin{theorem}
	$|R^w_n(\td)| = C_n$, the $n$-th Catalan number.
	\label{thm:td_weak}
\end{theorem}
This result, as well as its proof via binary trees (Proof~2 below),
was first found by Aaron Williams~\cite{Williams}.
We~provide several proofs:
first we derive the generating function for $R^w_n(\td)$,
and then we construct bijections between $R^w_n(\td)$ and several standard Catalan structures, 
such as binary trees, Dyck paths, $(213)$-avoiding permutations, 
and $(10)$-avoiding inversion sequences. 
These proofs emphasize different aspects of the result:
bijections to binary trees and Dyck paths demonstrate that 
$\td$-avoiding weak rectangulations is a very natural Catalan structure;
the bijection with $(213)$-avoiding permutations shows how 
adding the pattern $\td$ restricts the well-known bijection between
weak rectangulations and Baxter permutations;
finally, the bijection $\tau$ with $(10)$-avoiding inversion sequences 
is used in Section~\ref{sec:td_strong1} as a basis for 
bijections between $\td$-avoiding \textit{strong} rectangulations 
and several classes of pattern-avoiding inversion sequences.

\medskip
\nid \textbf{Proof 1: Generating functions.}
In this proof we consider a structural decomposition of $\td$-avoiding rectangulations,
and translate it into an equation satisfied by the generating function. 
It is an adaptation of the proof of the claim that 
(all) weak rectangulations are enumerated by Schröder numbers 
(see~\cite[Theorem~2]{AckermanBarequetPinter2006num} and~\cite[Proposition~2.4]{AsinowskiCardinalFelsnerFusy2024}).

Let $\mathcal{R} \in R^w(\td)$. 
Note that $\mathcal{R}$ is necessarily guillotine, since both windmills contain $\td$.
Every guillotine rectangulation of size $>1$ has either horizontal or vertical cut(s)
and is called accordingly horizontal or vertical.
Let $R(x)$, $H(x)$, and $V(x)$ be the generating functions (with respect to the size) 
for all, horizontal, and vertical rectangulations in $R^w(\td)$, respectively. 
Then we have
\begin{equation}
R(x) = x + H(x) + V(x).
\label{eq:R}
\end{equation}

\begin{figure}[h]
	\centering
	\includegraphics[scale=0.9]{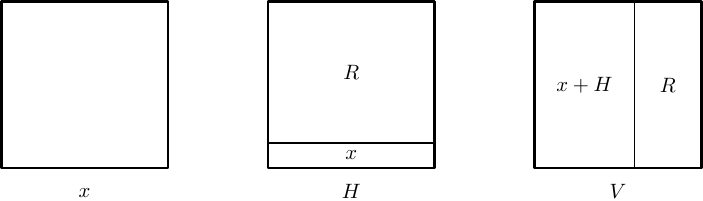}
	\caption{The structural decomposition of an element of $R^w(\td)$.}
	\label{fig:guillotine_decomp}
\end{figure}

Consider the structural decomposition of horizontal and vertical rectangulations in $R^w(\td)$ (see Figure~\ref{fig:guillotine_decomp}). 
Given a horizontal rectangulation in $R^w(\td)$, let $s$ be its lowest cut. 
Above $s$ there can be any $\td$-avoiding rectangulation;
below $s$ there can be a only a single non-partitioned rectangle.
Hence, we have
\begin{equation}
H(x) = x \, R(x).
\label{eq:H}
\end{equation}

Given a vertical rectangulation in $R^w(\td)$, let $t$ be its left-most cut. 
To the left of $t$ there can be either a~rectangulation of size $1$ or any $\td$-avoiding horizontal rectangulation; to the right of $t$ there can be any $\td$-avoiding rectangulation. 
Hence, we have
\begin{equation}
V(x) = \big(x+H(x)\big) \, R(x).
\label{eq:V}
\end{equation}

Solving \eqref{eq:R}, \eqref{eq:H}, and \eqref{eq:V} for $R(x)$, we obtain
\[xR^2(x)+(2x-1)R(x)+x=0,\]
which yields
\[R(x) = \frac{1-\sqrt{1-4x}}{2x}-1,\]
the generating function of Catalan numbers (starting at $C_1$).\hfill$\qed$

\medskip
\nid \textbf{Proof 2: Bijection with binary trees and staircases.}
As mentioned above, this proof was first given by Aaron Williams~\cite{Williams}.
It employs the fact that every diagonal rectangulation can be decomposed into
two binary trees by cutting along the diagonal: see Figure~\ref{fig:trees}(a),
where the \textit{lower tree} is red and the \textit{upper tree} is blue.
More precisely, there is a bijection between weak rectangulations 
and \textit{twin-binary trees} --- pairs of binary trees that satisfy a~condition to match the 
orientations of interior leaves~\cite{YaoChenChengGraham2003}. Below we show that \textit{$\td$-avoiding} weak rectangulations
are in bijection with \textit{individual} binary trees.

\begin{figure}[h]
\begin{center}
\includegraphics[scale=0.75]{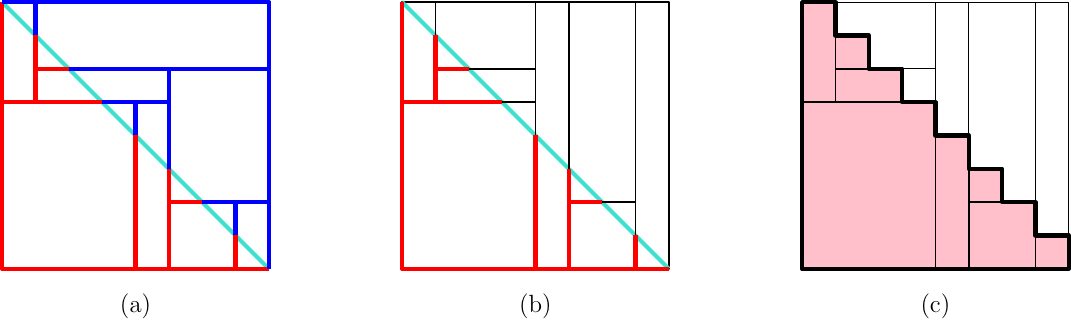} 
\end{center}
\caption{Illustration to the second proof of Theorem~\ref{thm:td_weak}.
(a) A diagonal rectangulation is uniquely determined by a pair of binary trees.
(b) A $\td$-avoiding diagonal rectangulation is uniquely determined by a single binary tree.
(c)~A~$\td$-avoiding diagonal rectangulation is uniquely determined by a Catalan staircase (pink).}
\label{fig:trees}
\end{figure}

Given a $\td$-avoiding weak rectangulation $\R$, consider its diagonal drawing.
Since $\td$ is avoided, all the joints that lie above the diagonal are of the shape $\tl$. 
Therefore, the upper binary tree --- and hence, the entire rectangulation --- 
is entirely determined by the lower binary tree.
This leads to the bijection which maps~$\mathcal{R}$ to the lower binary tree obtained in this way.
For a binary tree $\mathcal{T}$, the corresponding $\td$-avoiding diagonal rectangulation can be constructed as follows (see Figure~\ref{fig:trees}(b)):
\begin{enumerate}
	\item Embed $\mathcal{T}$ in a rectangle $R$ so that its root is in the \textrm{SW}-corner,
	all the edges are horizontal or vertical and non-crossing,
	 and all the leaves lie on the \textrm{NW}--\textrm{SE} diagonal.
	\item Extend every vertical edge upwards until it meets \textrm{N} --- the top side of $R$. 
	\item Extend every horizontal branch rightwards until it meets a vertical segment or \textrm{E} --- the right side of $R$.
\end{enumerate}
Since binary trees are enumerated by Catalan numbers~\cite[Chapter 2, item~5]{Stanley2015}, the result follows.

In fact, a minor modification of the subdiagonal part of $\td$-avoiding rectangulations
recovers \textit{Catalan staircases}. They are defined 
in~\cite[Chapter 2, item~205]{Stanley2015} as ``tilings of the staircase shape [with $n$ steps] by $n$ rectangles'',
and they are clearly equivalent to binary trees: see Figure~\ref{fig:trees}(c).\hfill$\qed$

\medskip
\nid \textbf{Proof 3: Bijection with non-decreasing inversion sequences and Dyck paths.}
We establish a size-preserving bijection $\tau$ from $\td$-avoiding weak rectangulations
to non-decreasing inversion sequences, which then directly yields a~bijection with Dyck paths. 
Note that the family of non-decreasing inversion sequences is precisely $I(10)$.
Every~non-decreasing inversion sequence $e$ consists of \textit{plateaus} --- 
maximal consecutive subsequences with the same value. 
The first elements of the plateaus are precisely the left-to-right maxima of $e$.
Consecutive plateaus are separated by a \textit{jump} --- a pair of consecutive elements $e_i, e_{i+1}$ 
with $e_i < e_{i+1}$; the \textit{height} of such a jump is $e_{i+1}-e_i$.

Let $\mathcal{R} \in R^w(\td)$. 
We construct a non-decreasing inversion sequence $\tau(\mathcal{R})$ as follows:
\begin{enumerate}
\item For each rectangle $X$ of $\mathcal{R}$, 
label it with the number of rectangles to the left of $X$. 
This labeling will be called the \textit{L-labeling}.
\item Consider also the \textrm{SW}--\textrm{NE} ordering of the rectangles of $\R$.
\item We define $\tau(\R)$ to be the sequence $e$ obtained by taking the rectangles according to the \textrm{SW}--\textrm{NE} ordering and reading their L-labels.
\end{enumerate}
These steps are demonstrated in Figure~\ref{fig:tau_illustration}(a,b,c).
For the rectangulation $\R$ in this example
we obtain $\tau(\R) = (0,0,1,1,1,4,4,4,4,4,8,8,12,12,12,12,12,12)$.

\begin{figure}[h]
	\centering
	\includegraphics[width=\textwidth]{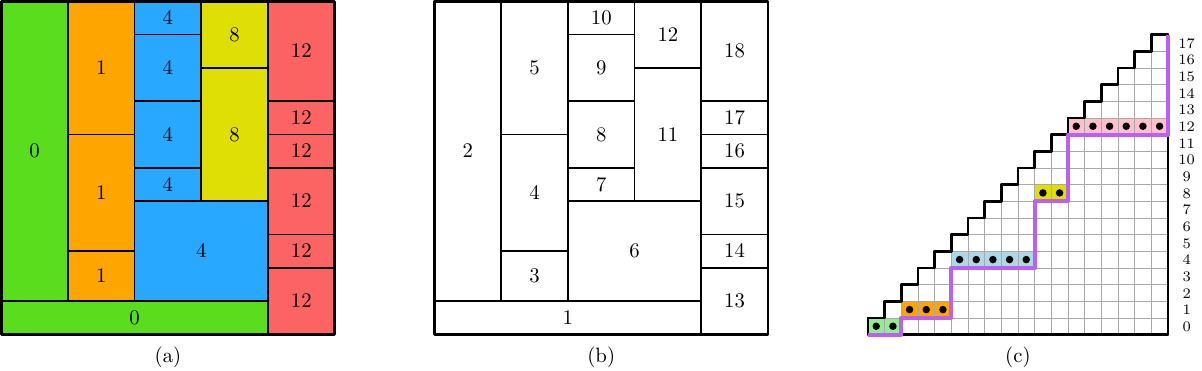}
	\caption{Bijection $\tau \colon R^w(\td) \to I(10)$:
	(a) Rectangulation $\mathcal{R}$ with the L-labeling;
	(b)~Rectangulation $\mathcal{R}$ with the \textrm{SW}--\textrm{NE} labeling;
(c) The inversion sequence $e=\tau(\mathcal{R})$ and the Dyck path $p=\delta(\R)$.	 
	 }
	\label{fig:tau_illustration}
\end{figure}
To see that $e$ is an inversion sequence,
note that if $Y_k \in \R$ is the rectangle labeled $k$ in the \textrm{SW}--\textrm{NE} labeling,
then there are \textit{precisely} $k-1$ rectangles which lie left of or below $Y_k$,
and therefore, \textit{at most} $k-1$ rectangles left of $Y_k$:
hence, $e_k \leq k-1$. 
To see that $e$ is non-decreasing,
consider $Y_{k}$ and $Y_{k+1}$, the rectangles of $\R$ labeled (respectively) $k$ and $k+1$
in the \textrm{SW}--\textrm{NE} labeling. Then $Y_{k+1}$ is either to the \linebreak

\vspace{-9pt}
\noindent\begin{minipage}{0.69\textwidth}
{right of or above $Y_k$.
In the former case the L-label of $Y_k$ is clearly smaller than that of $Y_{k+1}$, which means $e_k < e_{k+1}$.
In the latter case there is a horizontal segment $s$ which contains the top side of~$Y_{k}$ and the bottom side of~$Y_{k+1}$. 
The left endpoint of $s$ is the \textrm{SW}-corner of $Y_{k+1}$ and,  since $\R$ avoids $\td$,
it is also the \textrm{NW}-corner of~$Y_{k}$. Hence, the set of rectangles left of $Y_{k}$ and the set of rectangles left of $Y_{k+1}$ \linebreak}
\end{minipage}
\hspace{9pt}
\begin{minipage}{0.23\textwidth}{\raisebox{-6pt}{\includegraphics[scale=1]{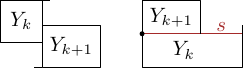}} }
\end{minipage}

\vspace{-8pt}
\noindent are the same.
Therefore, in this case, the L-labels of $Y_{k}$ and of $Y_{k+1}$ are equal, and we have $e_k = e_{k+1}$.

The result $|I_n(10)|=C_n$ was shown in several contributions: see, for example,
\cite[Theorem~27]{MartinezSavage2018} and~\cite[Section 2.1]{Testart2024}. 
The following direct bijection $\varepsilon$
from non-decreasing inversion sequences to Dyck paths 
is given in~\cite[Chapter~2, item~78]{Stanley2015}:
Given $e \in I_n(10)$,
consider its plot drawn over in the $[0,n]\times[0,n]$ grid.
Connect the point $(0,0)$ to the point $(n,n)$ by the lattice path which consists of $(1,0)$-steps just below the cells marked in the plot, and $(0,1)$-steps which are then needed to complete such a path. This lattice path is a (subdiagonal) Dyck path $\varepsilon(e)$ of semilength $n$,
and it is clear that $\varepsilon$ is bijective. 
Then $\delta := \varepsilon \circ \tau$
is a mapping from $R^w(\td)$ to Dyck paths of semilength $n$.
In Figure~\ref{fig:tau_illustration}(c), the Dyck path $p=\delta(\R)$ is shown by a purple line.

We now make several observations that will be useful later.
Denote the left side, the vertical segments, and the right side of $\R$,
ordered from left to right, by $s_1, s_2, \ldots, s_{t+1}$.
Two rectangles have the same L-label if and only if they are right neighbors of the same $s_k$.
When we construct $\tau(\R)$ by taking the rectangles of $\R$ according to the \textrm{SW}--\textrm{NE} ordering,
we first take the right neighbors of $s_1$, then the right neighbors of $s_2$, and so on;
for~each $k=1, 2, \ldots, t$, the right neighbors of $s_k$ are taken from bottom to top.
Therefore, the length of the $k$-th plateau is equal to the number of rectangles that neighbor $s_{k}$ on the right.
Finally, for $k \geq 2$, the set of rectangles to the left of $s_k$
is the disjoint union of rectangles to the left of $s_{k-1}$
and the left neighbors of $s_k$.
Therefore, the height of the jump from the $(k-1)$-th to the $k$-th plateau
is equal to the number of left neighbors of $s_k$.

As a result, the Dyck path $\delta(\R)$ can be directly obtained as follows:
Given $\R\in R^w(\td)$, we traverse it as shown by the white arrow in Figure~\ref{fig:dyck}(a): 
beginning in the \textrm{SW}-corner and ending in the \textrm{SE}-corner, 
we trace the left, right, and top sides of $R$ (\textrm{W}, \textrm{E}, and \textrm{N}) and each of the vertical segments in $\R$. 
Each rectangle we cross as we trace upwards corresponds to a $(1,0)$-step in the Dyck path 
and each rectangle we cross as we trace downwards corresponds to a $(0,1)$-step in the Dyck path. 
If we use instead $(1,1)$-and $(1,-1)$-steps, we directly obtain the standard form of~$p$.
Then, for every rectangle $X$ of $\R$, the left and the right sides of~$X$ correspond to
a matching pair of steps in $p$, see Figure~\ref{fig:dyck}(b).

\begin{figure}[h]
\begin{center}
\includegraphics[width=\textwidth]{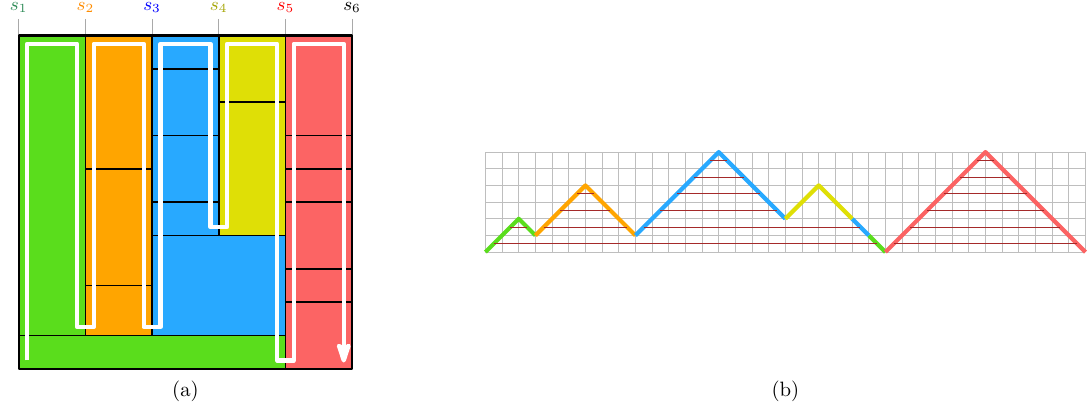} 
\caption{A direct bijection between $\td$-avoiding weak rectangulations and Dyck paths.
Pairs of vertical sides of the same rectangle in $\R$ correspond to pairs of matching steps in $p$.}
\label{fig:dyck}
\end{center}
\end{figure}

Given a Dyck path $p$, it is easy to restore the unique rectangulation $\R$ such that $\delta(\R)=p$:
one can use the bounding box of $p$ as $R$, draw a vertical segment upwards from every valley,
and then the pairs of matching steps indicate how rectangles should be inserted between the 
vertical sides and segments of~$R$. This proves that $\delta$ is a~bijection; therefore $\tau$ is also a bijection. \hfill$\qed$

\medskip
The \textbf{fourth proof} connects the avoidance of $\td$ to permutation patterns.
Let $\beta$ be the following version of the bijection
given by Ackerman, Barequet, and Pinter~\cite{AckermanBarequetPinter2006bij}, 
between weak rectangulations and \textit{Baxter permutations}:
Given a weak rectangulation~$\R$ of size $n$, 
label its rectangles by $1, 2, \ldots, n$ according to the \textrm{SE}--\textrm{NW} ordering,
and then read these labels according to the the \textrm{SW}--\textrm{NE} ordering.%
\footnote{In~\cite{AckermanBarequetPinter2006bij},
one labels rectangles by to the \textrm{NW}--\textrm{SE} ordering,
and then reads the labels according to the the \textrm{SW}--\textrm{NE} ordering.
We use different orderings in order to make the statement of Theorem~\ref{thm:thetabeta} particularly simple.} The permutation obtained in this way is defined to be~$\beta(\R)$.

We prove a substantially stronger result which does not just deal with enumeration, 
but provides a correspondence between a rectangulation pattern and a permutation pattern.
Therefore, we give it separately as Theorem~\ref{thm:213}.
It implies $|R^w(\td)|=C_n$ directly, since $(213)$-avoiding permutations are enumerated by Catalan numbers (and all $(213)$-avoiding permutations are Baxter).

\begin{theorem}\label{thm:213}
A weak rectangulation $\R$ avoids the pattern $\td$ if and only the permutation $\pi=\beta(\R)$ avoids the pattern $213$.
\end{theorem}

\noindent 
\begin{minipage}{0.7\textwidth}
{
\noindent \textbf{Proof:} Assume that $\R$ contains $\td$, 
and denote the rectangles adjacent to its joint point by $X_i, X_j, X_k$ as in the drawing. 
Then, in the \textrm{SE}--\textrm{NW} ordering we have $X_i \prec X_j  \prec  X_k$,
and in the \textrm{SW}--\textrm{NE} ordering we have $X_j  \prec X_i  \prec  X_k$. Hence, $\pi$ has an occurrence of~$213$.
}
\end{minipage} \hspace{30pt}
\begin{minipage}{0.18\textwidth}
{
\includegraphics[scale=1]{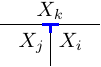}
}
\end{minipage}

\vspace{2pt}

\noindent 
\begin{minipage}{0.7\textwidth}
\setlength{\parindent}{1.5em}
{Now assume that $\pi$ contains the pattern $213$.
Then it necessarily contains $2\underline{13}$ ---
an occurrence of $213$ in which the elements corresponding to $1$ and $3$ are adjacent.
Therefore, $\R$ has three rectangles labeled $X_i, \, X_j, \, X_k$
such that we have $X_i\prec X_j \prec X_k$ in the \textrm{SE}--\textrm{NW} ordering,
$X_j \prec X_i \prec X_k$ in the \textrm{SW}--\textrm{NE} ordering, 
and, moreover, $X_k$ is the direct successor of~$X_i$ in the \textrm{SW}--\textrm{NE} ordering.
Then there is a horizontal segment $s$ that contains 
the top side of $X_i$ and the bottom side of $X_k$, and its left endpoint $P$ is the \textrm{SW}-corner of~$X_k$. 
If $X_i$ extends to the left so that its \textrm{NW}-corner is $P$,
then the left sides of $X_i$ and~$X_k$ are contained in the same vertical segment and
therefore, a rectangle $Y$ of $\R$ is to the left of~$X_i$ if and only if it is to the left of $X_k$. 
However, the \textrm{SE}--\textrm{NW} and \textrm{SW}--\textrm{NE} 
orderings imply that $X_j$ is to the left of $X_i$ but below~$X_k$.
Therefore, $X_i$~does not extend to contain~$P$, 
and hence, the left side of~$X_i$ together with $s$ forms $\td$. \hfill $\qed$
}
\end{minipage}\hspace{20pt}
\begin{minipage}{0.18\textwidth}
{
\includegraphics[scale=1]{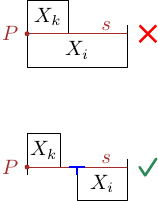}
}
\end{minipage}

\vspace{8pt}

The mappings $\tau$ and $\beta$ are related as the following theorem shows.

\begin{theorem}\label{thm:thetabeta}
For every weak rectangulation $\R$ we have $\tau(\R) = \Theta(\beta(\R))$.
\end{theorem}
\begin{proof}
Let $\pi = \beta(\R) = \pi_1 \pi_2 \ldots \pi_n$. This means: for every $j$,
the rectangle $Y_j$ labeled $j$ in the \textrm{SW}--\textrm{NE} labeling
is labeled $\pi_j$ in the \textrm{SE}--\textrm{NW} labeling.
Therefore, the number of indices $i$ such that $i<j$ and $\pi_i>\pi_j$
is equal to the number of rectangles that occur earlier than $Y_j$ in the \textrm{SW}--\textrm{NE} ordering
but later than $Y_j$ in the \textrm{SE}--\textrm{NW} ordering.
This is precisely the number of rectangles to the left of $Y_j$,
and therefore, the  $j$-th component of $\Theta(\beta(\R))$
is precisely the L-label of $Y_j$. Hence, $\Theta(\beta(\R))$ coincides with $\tau(\R)$.
\end{proof}

For example, for $\R$ from Figure~\ref{fig:tau_illustration},
we have 
$\beta(\R) = 7 \ 18 \ 15 \ 16 \ 17 \ 8 \ 11 \ 12 \ 13 \ 14 \ 9 \ 10 \ 1 \ 2 \ 3 \ 4 \ 5 \ 6 \ 7$,
whose inversion sequence is indeed $e=(0,0,1,1,1,4,4,4,4,4,8,8,12,12,12,12,12,12)$,
see Figure~\ref{fig:tau_perm}.

\begin{figure}[h]
	\centering
\includegraphics[width=\textwidth]{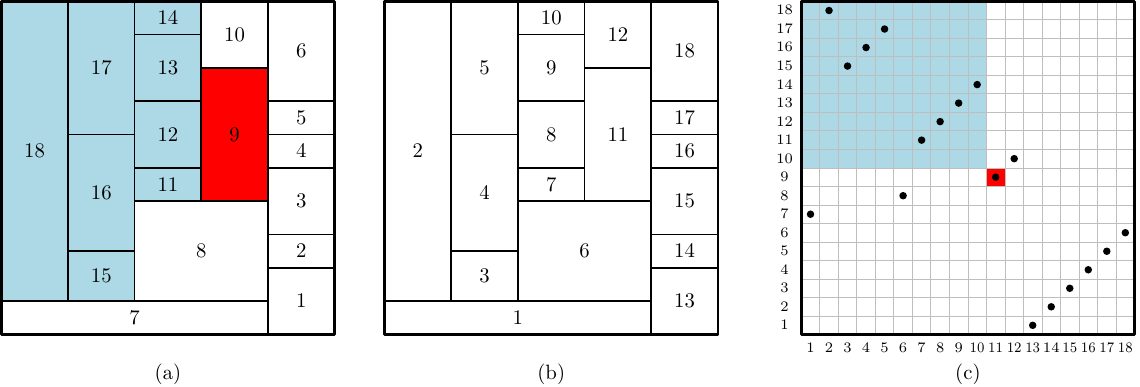}
	\caption{Illustration to Theorem~\ref{thm:thetabeta}: 
(a)~The \textrm{SE}--\textrm{NW} labeling of $\R$.	
(b)~The \textrm{SW}--\textrm{NE} labeling of $\R$.
(c)~The permutation $\beta(\R)$: $\Theta(\beta(\R))$ coincides with $\tau(\R)$. The highlighted areas illustrate $\pi_{11}=9$ and $e_{11}=8$.}
	\label{fig:tau_perm}
\end{figure}

\subsubsection{Classes of inversion sequences enumerated by A279555}
\label{sec:555}
Now we turn to the enumeration of $\td$-avoiding \textit{strong} rectangulations. 
M\"utze and Namrata~\cite{MuetzeNamrata} explored them computationally and 
observed that the first elements of their enumerating sequence match \href{https://oeis.org/A279555}{A279555}. 
In the OEIS, \href{https://oeis.org/A279555}{A279555} is defined as the enumerating sequence of two classes of inversion sequences, 
$I(010,110,120,210)$ and $I(010,100,120,210)$.
To our knowledge, these classes were first considered by Martinez and Savage
in their study of classes of inversion sequences 
that avoid triples of relations~\cite{MartinezSavage2018}.
In their classification, these are classes 764A and 764B;
in~\cite[Theorem~62]{MartinezSavage2018} they show that these classes are Wilf-equivalent.
At the time of writing, these two classes of inversion sequences are the only interpretations of 
\href{https://oeis.org/A279555}{A279555} in OEIS. 
However, this sequence has been proven, or conjectured, 
to enumerate some additional classes of inversion sequences, as outlined below.

In~\cite[Theorem~8.1]{YanLin2020}, Yan and Lin proved that 
$I(011, 201)$ and $I(011, 210)$ (their classes 3091A and 3091B) 
are Wilf-equivalent, and conjectured that they are also enumerated by 
\href{https://oeis.org/A279555}{A279555}~\cite[Conjecture~8.3]{YanLin2020}.
To the best of our knowledge, this conjecture had not been proven at the time of writing.
We provide a proof in Theorem~\ref{thm:conj}.

In~\cite[Section~4.5]{KotsireasMansourYildirim2024}, Kotsireas, Mansour, Yıldırım
constructed a generating tree for $I(011, 201)$ 
and gave a~functional equation satisfied by the generating function. 
The expression given in~\cite[Equation~4.16]{KotsireasMansourYildirim2024} contained a typo which was fixed by Pantone~\cite[Section~6.5]{Pantone2024}, 
who also computed 500 terms for $I(011, 201)$ and for $I(010, 100, 120, 210)$.
His computations showed that the enumerating sequences of these classes coincide 
at least to that extent --- thus, providing strong evidence to the conjecture of Yan and Lin.

In~\cite[Table~1, Case~166]{CallanMansour2023}, Callan and Mansour 
listed nine classes of inversion sequences determined by quadruples of patterns of length $3$,
which are enumerated, or were conjectured to be enumerated, 
by~\href{https://oeis.org/A279555}{A279555}. 
The classes $I(010,110,120,210)$ and $I(010,100,120,210)$
from the original definition of~\href{https://oeis.org/A279555}{A279555} 
are their classes 166(8) and 166(6).
The classes $I(011, 201)$ and $I(011, 210)$ appear in their list as 
$I(011, 101, 110, 201)$ (class 166(3)) and $I(011, 101, 110, 210)$ (class 166(4)).\footnote{To see that $I(011, 201) = I(011, 101, 110, 201)$
and $I(011, 210) = I(011, 101, 110, 210)$, note that avoidance of $011$ 
directly implies avoidance of $101$ and of $110$. In~\cite[Theorem~1]{CallanMansour2023}, a different proof is given.} 
Thus, Conjecture~1.2~\cite{CallanMansour2023} by Callan and Mansour is equivalent to Conjecture~8.3~\cite{YanLin2020} by Yan and Lin.  

From the results achieved in~\cite{CallanMansour2023, KotsireasMansourYildirim2024,MartinezSavage2018, Pantone2024, YanLin2020} 
it is known at the time of writing that:
\begin{itemize}
\item The classes $166(6-9)$ are Wilf-equivalent and enumerated by~\href{https://oeis.org/A279555}{A279555},  
\item The classes $166(1-5)$ are Wilf-equivalent and \textit{conjecturally} enumerated by~\href{https://oeis.org/A279555}{A279555}.
\end{itemize}

We summarize these results in Table~\ref{tab:nine},
which is adapted from~\cite[Table~1, Case~166]{CallanMansour2023}.

\renewcommand{\arraystretch}{1.2}
\begin{table}[h]
%\scalebox{0.95}{
\begin{center}
\begin{tabular}{|c|c|c|c|}
\hline
 & Class & Proofs of Wilf-equivalence & Notation in~\cite{MartinezSavage2018, YanLin2020}   \\ \hline
(1) & $I(010,100,110,201)$ &  &   \\ 
(2) & $I(010,100,110,210)$ & $(3) \sim (4)$: \cite{YanLin2020}.  &   \\ 
(3) & $I(011,101,110,201)$ & $(1-5)$: Isomorphic generating trees~\cite{CallanMansour2023}.  & $=I(011,201)$ \cite[3091A]{YanLin2020}  \\ 
(4) & $I(011,101,110,210)$ & \href{https://oeis.org/A279555}{A279555}: conjectured~\cite{CallanMansour2023, Pantone2024, YanLin2020}.& $=I(011,210)$ \cite[3091B]{YanLin2020}   \\ 
(5) & $I(010,101,110,201)$ &  & (shown in~\cite{CallanMansour2023}).   \\ \hline
(6) & $I(010,100,120,210)$ & $(6) \sim (8)$: \cite{MartinezSavage2018}.   & \cite[764B]{MartinezSavage2018}   \\ 
(7) & $I(010,101,120,201)$ & $(6-8)$: Isomorphic generating trees~\cite{CallanMansour2023}. &  \\ 
(8) & $I(010,110,120,210)$ & \href{https://oeis.org/A279555}{A279555}: (6,8) by definition~\cite{MartinezSavage2018}. & \cite[764A]{MartinezSavage2018}  \\ \hdashline
(9) & $I(010,101,120,210)$ & \href{https://oeis.org/A279555}{A279555}: via bijection to (7)~\cite{CallanMansour2023}.&    \\ \hline 
   \end{tabular}
%   }
   \caption{Nine classes of inversion sequences avoiding four patterns of length 3, 
   proven or conjectured to be enumerated by \href{https://oeis.org/A279555}{A279555} in~\cite{CallanMansour2023,MartinezSavage2018, Pantone2024,YanLin2020}.
   (Adapted from~\cite[Table~1, Case~166]{CallanMansour2023}.)}
   \label{tab:nine}
\end{center}
\end{table}
\renewcommand{\arraystretch}{1}

In this section we show that $R^s(\td)$ is equinumerous to classes 
166(3,6,7,8) from Table~\ref{tab:nine} by constructing generating trees and also 
by giving explicit bijections.
First, in Section~\ref{sec:td_strong1}, we show that $R^s(\td)$ is equinumerous
to $I(010,101,120,201)$, $I(010,110,120,210)$ and $I(010,100,120,210)$,
thus confirming the conjecture by Mütze and Namrata that $R^s(\td)$ is enumerated
by \href{https://oeis.org/A279555}{A279555}.
To the best of our knowledge, it is the first interpretation of 
this sequence using a combinatorial structure other than a class of inversion sequences.
Then, in Section~\ref{sec:td_strong2}, we show that $R^s(\td)$ is equinumerous to $I(011,201)$. 
In~Section~\ref{sec:conj} we combine these results, which completes the proof of the conjecture 
that all the classes in Table~\ref{tab:nine} are Wilf-equivalent and
enumerated by \href{https://oeis.org/A279555}{A279555}.

\subsubsection{$\td$-avoiding strong rectangulations: bijections with classes of inversion sequences 
$I(010,101,120,201)$, $I(010,110,120,210)$, and $I(010,100,120,210)$}
\label{sec:td_strong1}

Denote 
$I^{(6)}:=I(010,100,120,210)$,
$I^{(7)}:=I(010,101,120,201)$, and 
$I^{(8)}:=I(010,110,120,210)$, in accordance to the case number in Table~\ref{tab:nine}.
In this section we prove that $R^s(\td)$ is equinumerous to these three classes.
We first prove it by considering a generating tree T1 
given by Pantone~\cite{Pantone2024} for $I^{(6)}$,
and showing that it can be also regarded as a generating tree for $R^s(\td)$. 
Then we describe explicit bijections between $R^s(\td)$ and these three classes 
of inversion sequences.

We begin by exploring the structure of $I^{(7)}$, $I^{(8)}$, and $I^{(6)}$.
Let $e=(e_1, e_2, \ldots, e_n)$ be an element of any of these classes, and let 
$(0=) \ e_{a_1}<e_{a_2}<\ldots<e_{a_t}$ be its left-to-right maxima.
Then, since $e$ avoids $010$ and $120$, it satisfies the following condition:
\begin{itemize}
\item \textbf{Condition A}: For every $j=1, 2, \ldots, t$,
we have $a_j \leq k < a_{j+1} \ \Longrightarrow e_{a_{j-1}} < e_k \leq e_{a_j}$. \\
(For $j=1$, the condition just says that we have $e_k=0$ for all $k$ such that $(1 =) \ a_1 \leq k < a_2$.) 
\end{itemize}
This means that in the plot of $e$ all the elements lie in rectangular areas
$[a_j, a_{j+1}-1] \times [e_{a_{j-1}}+1, e_{a_{j}}]$ 
to which we refer as \textit{active areas}.
(By convention, the first active area is $[a_1, a_2] \times \{0\}$,
and the last active area is $[a_t, n] \times [e_{a_{t-1}}+1, e_{a_{t}}]$.)

Then, the two remaining patterns imply a condition, specific to each of the classes, 
concerning elements within active areas:

\begin{itemize}
\item For $I^{(7)} = I(010,101,120,201)$, \textbf{Condition B1}:
In every active area, the elements are weakly decreasing.
\item For $I^{(8)} = I(010,110,120,210)$, \textbf{Condition B2}:
In every active area, the elements are weakly increasing, 
except for the first element $e_{a_j}$. 
\item For $I^{(6)} = I(010,100,120,210)$, \textbf{Condition B3}: 
In every active area, the elements whose value is not equal to its first element $e_{a_j}$, 
are strictly increasing. \end{itemize}

It is easy to see that every sequence that satisfies condition \textbf{A} and any of the conditions
\textbf{B1}, \textbf{B2}, \textbf{B3}, belongs to the respective class.
That is, conditions \textbf{A} and \textbf{B1} characterize $I^{(7)}$,
conditions \textbf{A} and \textbf{B2} characterize $I^{(8)}$, and
conditions \textbf{A} and \textbf{B3} characterize $I^{(6)}$.
In Figure~\ref{fig:str1} we give three examples to exhibit these characterizations 
(active areas are highlighted by orange).
Note that these characterizations immediately imply 
$I^{(7)} \sim I^{(8)} \sim I^{(6)}$:
an element of $I^{(8)}$ can be obtained from an element of $I^{(7)}$
by vertical reflection of every active area, except for its first column; 
and an element or $I^{(6)}$ can be obtained from an element of $I^{(8)}$
by the following transformation:
if several elements in the $j$-th active area have the same value smaller than $e_{a_j}$,
then all of them, except for the last one, are replaced by $e_{a_j}$.
These transformations are clearly bijective; 
in Figure~\ref{fig:str1} we apply them to obtain $e_1 \mapsto e_2 \mapsto e_3$. 
We will primarily deal with $I^{(7)}$, 
since the condition \textbf{B1} appears to be the most natural one.

\begin{figure}[h]
\begin{center}
\includegraphics[width=\textwidth]{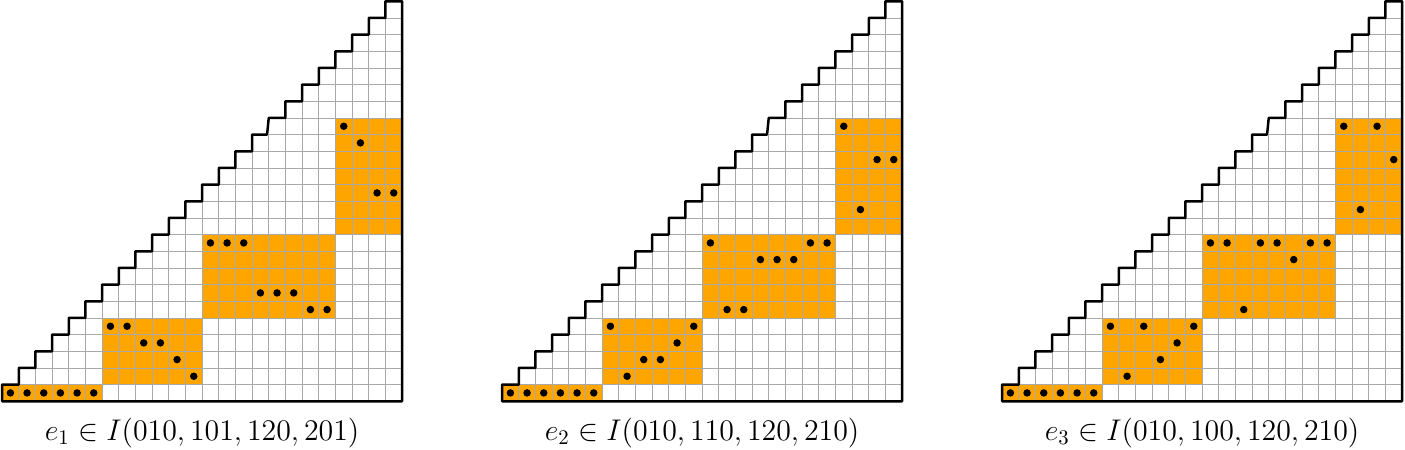} 
\end{center}
\caption{Examples of inversion sequences in $I^{(7)}=I(010,101,120,201)$, $I^{(8)}=I(010,110,120,210)$, and $I^{(6)}=I(010,100,120,210)$.
Active areas are highlighted by orange.}
\label{fig:str1}
\end{figure}

\bigskip

In~\cite[Section~6.5]{Pantone2024}, Pantone gave the following generating tree\footnote{See, for example, \cite{West1996} 
for a brief introduction to the method of generating trees and the related terminology.} T1 for $I^{(6)}$:

\begin{tcolorbox}[breakable,no shadow,
empty,
enhanced,]
\textbf{Generating Tree T1~\cite[Section~6.5]{Pantone2024}}
\[
\begin{array}{rrclccc}
\mathsf{Root:} & \  (1,0). & & & & & \\
\mathsf{Succession \ rules:} &(k, \ell) \ & \longrightarrow & (1,k-1), \ (2,k-2),  \ \ldots, \ (k,0); && (*)\\
&&& (k+1,0), \ (k+1, 1), \ \ldots, \  (k+1, \ell). && (**)
\end{array}
\]
\end{tcolorbox}
Here, $k$ is the bounce\footnote{Reminder: The bounce of an inversion sequence $e$ is 
$k=n-M$, where $n$ is the length and $M$ is the maximum value of $e$.} of $e \in I^{(6)}_n$, and 
$\ell$ is the number of \textit{admissible values}
(that is, the values that can be inserted at the end of $e$ to produce a valid element of $I^{(6)}_{n+1}$) which are smaller than $M$, the maximum of $e$.
The~rule $(*)$ represents adding a new maximum (and starting a new active area).
The rule~$(**)$ represents adding a new value within the last active area of $e$.
We refer to the pair $(k, \ell)$ for a specific $e$ as the \textit{type} of $e$.

It is easy to verify that T1 is also a generating tree for $I^{(7)}$ and $I^{(8)}$. 
In both cases, the statistic $k$ has the same meaning as for $I^{(6)}$ (the bounce).
The statistic $\ell$ for $I^{(8)}$ has the same meaning as for $I^{(6)}$;
and the statistic~$\ell$ for $I^{(7)}$ is the number of 
admissible values which are smaller than the final value of $e$. 
(All such values clearly belong to the range of the last active area.)
See Figure~\ref{fig:gt1is} for the illustration:
assuming that a given inversion sequence $e \in I^{(7)}_n$ has the type $(k, \ell)$,
black dots show all admissible values,
and the type of $e' \in I^{(7)}_{n+1}$ which is obtained by adding the respective admissible value
is shown near every dot.

\begin{figure}[h]
\begin{center}
\includegraphics[scale=1]{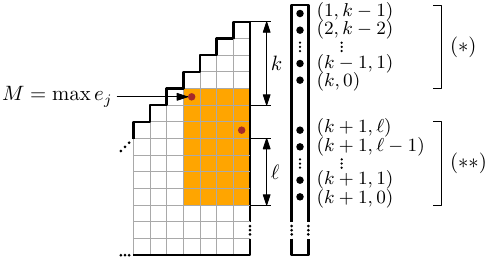}
\caption{Illustration for the succession rules of T1 for $I^{(7)}$.}
\label{fig:gt1is}
\end{center}
\end{figure}

\medskip

Now we can state and prove the main result of this section.

\begin{theorem}
\label{thm:strong1}
For every $n \geq 1$ we have
\[|R_n^s(\td)|=|I_n(010,101,120,201)|=|I_n(010,110,120,210)|=|I_n(010,100,120,210)|.\]
\end{theorem}
\begin{proof}
We show that T1 is a generating tree for $R^s(\td)$,
where the interpretation of $k$ and $\ell$ is as follows:
\begin{itemize}
\item $k$ is the number of $\mathrm{E}$-rectangles in $\R$,
\item $\ell$ is the number of horizontal segments that neighbor 
the $\mathrm{NE}$-rectangle of $\R$ on the left.
\end{itemize}

We refer to the pair $(k, \ell)$ for a specific $\R$ as the \textit{type} of $\R$.
The rectangulation of size $1$ has type $(1,0)$. 
Given a rectangulation $\R \in R^s_n(\td)$ of type $(k, \ell)$, 
we can obtain a rectangulation $\R' \in R^s_{n+1}(\td)$  
by inserting a~new $\mathrm{NE}$-rectangle in one of the following ways (see Figure~\ref{fig:gt1}):
\begin{itemize}
\item Inserting a new $\mathrm{NE}$-rectangle while pushing to the left the upper $j$ $\mathrm{E}$-rectangles, where $1 \leq j \leq k$.
The new rectangulation has the type $(k-j+1, j-1)$. This yields the succession rule $(*)$.
\item Inserting a new $\mathrm{NE}$-rectangle while pushing downwards the top edge of the  $\mathrm{NE}$-rectangle of $\R$, 
such that precisely $i$ segments, where $0 \leq i \leq \ell$, touch the new $\mathrm{NE}$-rectangle from the left.
The new rectangulation has the type $(k+1, i)$.  This yields the succession rule $(**)$.
\end{itemize}

\begin{figure}[h]
\begin{center}
\includegraphics[width=\textwidth]{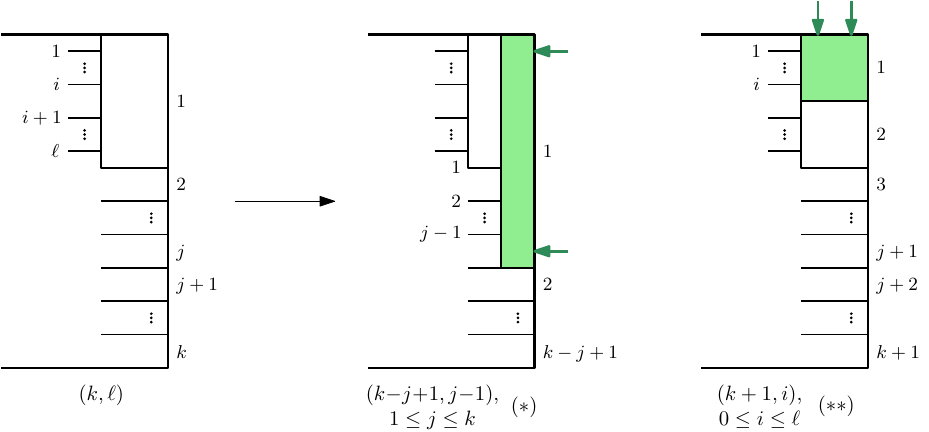}
\caption{Illustration for the succession rules of T1 for $R^s(\td)$.}
\label{fig:gt1}
\end{center}
\end{figure}

These succession rules ensure that every rectangulation in $R^s(\td)$ is generated precisely once:
Consider $\R' \in R^s_{n+1}(\td)$, and let $X$ be its $\mathrm{NE}$-rectangle.
If the $\mathrm{SW}$-corner of $X$ has the $\tu$ shape or $X$ is an $\mathrm{S}$-rectangle,
then $\R'$ is obtained from $\R\in R^s_{n}(\td)$ by rule $(*)$.
If the $\mathrm{SW}$-corner of $X$ has the $\tr$ shape or $X$ is a $\mathrm{W}$-rectangle,
then $\R'$ is obtained from $\R\in R^s_{n}(\td)$ by rule $(**)$.
In both cases $\R$ obtained from $\R'$ by the ($\mathrm{NE}$) \textit{corner rectangle deletion},
and $\R'$ from $\R$ by the ($\mathrm{NE}$) \textit{corner rectangle insertion}~\cite{AckermanBarequetPinter2006bij,Hong2000}.
Thus, we refer to constructing a~rectangulation from the size-$1$ rectangulation by these succession rules 
as \textit{constructing by $\mathrm{NE}$ insertion}.

Thus, T1 is a generating tree for $R^s(\td)$.
Since it is also a generating tree for $I^{(6)}$, for $I^{(7)}$, and for $I^{(8)}$, the theorem follows.
\end{proof}

Let $\tau^{(7)}:R^s(\td) \to I^{(7)}$ be the bijection that  
respects the succession rules of T1 for these two classes,
that is: $\tau^{(7)}(\R) = e$ if and only if $\R$ is obtained from the size-$1$ rectangulation 
and $e$ is obtained from the size-$1$ inversion sequence by the same sequence of succession rules.
Comparing how the corresponding succession rules of the generating tree T1
act on elements of $I^{(7)}$ and of $R^s(\td)$,
it is easy to derive a \textit{direct} description of $\tau^{(7)}$,
which modifies the bijection~$\tau \colon  R^w(\td) \to I(10)$ 
from the third proof of Theorem~\ref{thm:td_weak}. 
Let $\mathcal{R}\in R^s(\td)$, and let $\bar e=\tau(\mathcal{\bar R})$,
where $\bar \R$ is the weak rectangulation corresponding to $\R$. 
Since $\bar e$ is non-descending, it belongs to $I^{(7)}$, and its active areas are defined as above.
To account for the shufflings along vertical segments in $\mathcal{R}$, we modify $\bar e$ as follows:
For every~$k$, $2 \leq k \leq t$, consider the neighboring rectangles of the vertical segment $s_k$,
and if the~$\mathrm{SW}$-corner of the $i$-th (bottom to top) right neighbor
touches the $j$-th (bottom to top) left neighbor, we decrease the value of
$\bar e_{a_k+i-1}$ (the $i$-th element of the $k$-th plateau) by $j-1$.
Since the jump from the $(k-1)$-th to the $k$-th plateau of $\bar e$
is the number of rectangles that neighbor $s_k$ on the left,
the new value is larger than $\bar e_{a_{k-1}}$ --- that is,
the modified value still belongs to the $k$-th active area.
Finally, it is readily seen that the modified elements of every active area are weakly decreasing.

To visualize this construction, we label, for every $s_k$,
the left neighboring rectangles by $a, b, c, \ldots$ top to bottom, 
and the right neighboring rectangles by $1, 2, 3, \ldots$  bottom to top, 
see Figure~\ref{fig:tau_strong1}(a).
In the plot of $\bar e$ we label, for every active area,
its rows $a, b, c, \ldots$ from bottom to top,
and its columns by $1, 2, 3, \ldots$ from left to right.
Then~we shift the dots of the plateau so that we have, 
with respect to this local labeling, the dot $(v,w)$
if and only if 
the \textrm{SW}-corner of the right neighbor of $s_k$ labeled $v$
lies on the boundary of the left neighbor of $s_k$ labeled~$w$.
Then~$\tau^{(7)}(\R)$ is the inversion sequence $e$ obtained in this way,
see Figure~\ref{fig:tau_strong1}(b) for an example. 
\begin{figure}[h]
\centering
\includegraphics[scale=0.9]{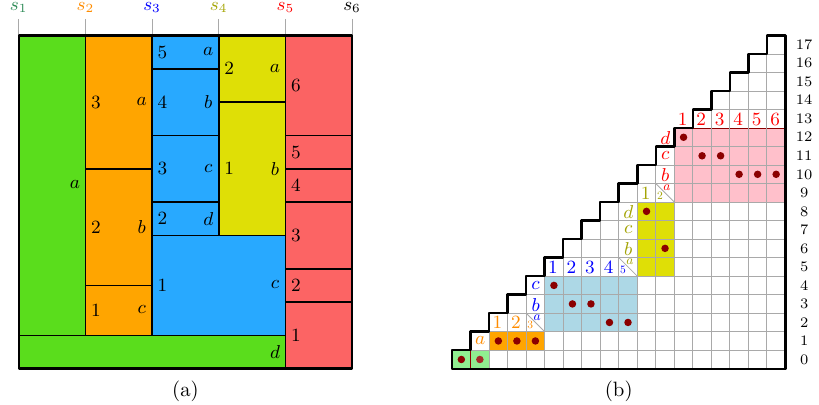}
\caption{Illustration to Theorem~\ref{thm:strong1}. 
(a) Rectangulation $\mathcal{R} \in R^s(\td)$. 
(b) $\tau^{(7)}(\R) \in I(010,101,120,201)$.}
\label{fig:tau_strong1}
\end{figure}

\smallskip

For bijections $\tau^{(8)} \colon  R^s(\td) \to I^{(8)}$
and $\tau^{(6)} \colon  R^s(\td) \to I^{(6)}$
we apply modifications of active areas as shown in Figure~\ref{fig:str1}.
The visualization $\tau^{(7)}$ in terms of $(v,w)$
can be adjusted accordingly, we omit the details.
Figure~\ref{fig:tau_strong2} shows $\tau^{(8)}(\R)$
and $\tau^{(6)}(\R)$ for the same rectangulation $\R$ as in 
Figure~\ref{fig:tau_strong1}.

\begin{figure}[h]
\centering
\includegraphics[scale=0.9]{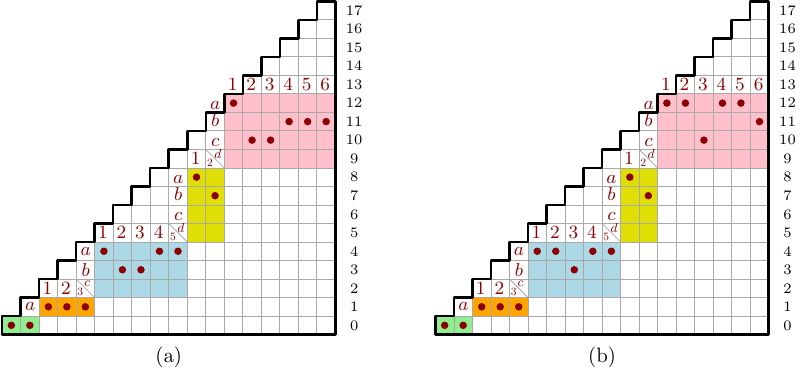}
\caption{Illustration to Theorem~\ref{thm:strong1}. 
(a) $\tau^{(8)}(\R) \in I(010,110,120,210)$. 
(b) $\tau^{(6)}(\R) \in I(010,100,120,210)$.
}
\label{fig:tau_strong2}
\end{figure}

Finally, we observe which statistics in inversion sequences
correspond to the numbers of rectangles that touch the sides of a rectangulation,
under all three bijections $\tau^{(6)}$, $\tau^{(7)}$, $\tau^{(8)}$.
These results are directly seen from the explicit description of the bijections, and we omit the proof.
\begin{prop}\label{thm:stat1}
Let $\R \in R^s(\td)$, and let $e$ be the image of $\R$ under $\tau^{(6)}$, $\tau^{(7)}$, or $\tau^{(8)}$. Then:
\begin{itemize}
\item The number of rectangles of $\R$ that touch $\mathrm{N}$ is the number of left-to-right maxima of $e$.
\item The number of rectangles of $\R$ that touch $\mathrm{E}$ is the bounce of $e$.
\item The number of rectangles of $\R$ that touch $\mathrm{S}$ is the number of high elements in $e$.
\item The number of rectangles of $\R$ that touch $\mathrm{W}$ is the number of $0$ elements in $e$.
\end{itemize}
\end{prop}

\subsubsection{$\td$-avoiding strong rectangulations: bijection with the class of inversion sequences $I(011,201)$}
\label{sec:td_strong2}

In this section we prove that $R^s(\td)$ is equinumerous to $I(011, 201)$.
We use $R^s(\tu)$ instead of $R^s(\td)$, since this leads to a particularly 
clear visual description.
Similarly to the previous section, we first prove the result by considering a generating tree T2
(different from T1) given by Pantone~\cite{Pantone2024} for $I(011, 201)$ 
and showing that it can be also regarded as a generating tree for $R^s(\tu)$. 
Then we provide an explicit size-preserving bijection between $R^s(\tu)$ and $I(011, 201)$.

\medskip

In~\cite[Section~6.5]{Pantone2024}, Pantone gave the following generating tree T2 for $I(011,201)$:

\begin{tcolorbox}[breakable,no shadow,empty,enhanced,]
\textbf{Generating Tree T2~\cite[Section~6.5]{Pantone2024}}
\[
\begin{array}{rrclccc}
\mathsf{Root:} & \  (1,0). & & & & & \\
\mathsf{Succession \ rules:}&(k, \ell) \ & \longrightarrow  & (1,k+\ell-1), \ (2,k+\ell-2),  \ \ldots, \ (k,\ell); & & (*)\\
&& & (k+1,0), \ (k+1, 1), \ \ldots, \ (k+1, \ell-1); && (**) \\
&&& (k+1,0). && (***)
\end{array}
\]
\end{tcolorbox}
In this case, $k$ is the bounce of $e \in I_n(011,201)$, and 
$\ell$ is the number of 
admissible values which are greater than $0$ and smaller than $M$, the maximum value of~$e$.
The rule~$(*)$ represents adding a new maximum.
The~rule~$(**)$ represents adding an admissible value 
greater than $0$ and smaller than $M$
(since the patterns $011$ and $201$ are avoided, 
such a new value is necessarily smaller than the final element of $e$).
The rule $(***)$ represents adding a $0$ element. 
(This explanation is adapted from~\cite{Pantone2024}.)
See Figure~\ref{fig:gt2is} for illustration;
white dots represent all the admissible values that can be added to $e$,
and the pairs to the right of corresponding black dots show which type is obtained if this admissible value is added. 
Note that all the values from $M+1$ to $n$ are always admissible for $(*)$,
and $0$ is always admissible for $(***)$.
A~precise characterization of elements admissible for $(**)$
(denoted by $v_1, v_2, \ldots, v_\ell$ in Figure~\ref{fig:gt2is})
will be provided in Proposition~\ref{thm:admissible2}.

\begin{figure}[h]
\begin{center}
\includegraphics[scale=1]{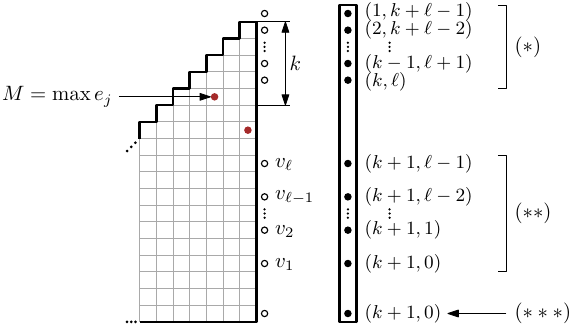}
\caption{Illustration for the succession rules of T2 for $I(011,201)$.}
\label{fig:gt2is}
\end{center}
\end{figure}

Before dealing with rectangulations, we introduce certain geometric relations.
If $X$ and $Y$ are two segments, or two rectangles, or one segment and one rectangle,
we say that one of them \textit{lies higher than / lies lower than / lies to the left of / lies to the right of} the other if there is a horizontal or vertical line that (weakly) separates them accordingly.
(Note that these relations are different from the relations \textit{above}, \textit{below}, etc., used earlier.)

\begin{theorem}\label{thm:strong2}
For every $n\geq 1$ we have $|R^s(\tu)| = |I(011, 201)|$.
\end{theorem}

\begin{proof} We say that a $\td$ joint is \textit{active} if the right endpoint of its
horizontal segment lies on $\mathrm{E}$.

We show that T2 is a generating tree for $R^s(\tu)$,
where the interpretation of $k$ and $\ell$ is as follows:
\begin{itemize}
\item $k$ is the number of $\mathrm{N}$-rectangles of $\R$,
\item $\ell$ is the number of active $\td$ joints.
\end{itemize}

The rectangulation of size $1$ has type $(1,0)$. 
Given a rectangulation $\R \in R^s_n(\tu)$ of type $(k, \ell)$, 
we can obtain a rectangulation $\R' \in R^s_{n+1}(\tu)$ 
by inserting a new $\mathrm{NE}$-rectangle 
in one of the following ways (see Figure~\ref{fig:gt2re}): 
\begin{itemize}
\item Inserting a new $\mathrm{NE}$-rectangle $X$ while pushing downwards 
the $j$ right-most $\mathrm{N}$-rectangles, where $1 \leq j \leq k$,
so that all the horizontal segments that exist in $\R$
are lower than the bottom side of $X$ (in particular, $X$~has no left neighboring segments).
This decreases the number of $\mathrm{N}$-rectangles by $j-1$
and increases the number of active $\td$ joints by $j-1$; hence,
$\R'$ has the type $(k-j+1, \ell+j-1)$. This yields the succession rule $(*)$.
\vspace{-70pt}
\item 
\begin{minipage}[t]{0.52\textwidth} 
Inserting a new $\mathrm{NE}$-rectangle $X$ by the following operation:
Given an active $\td$ joint, 
let $v$ be its vertical segment and
let~$r$ be the right part of its horizontal segment (see the image).
Push $r$ slightly downwards so that all the horizontal segments that were lower than $r$ in $\R$
are still lower than the shifted $r$ in $\R'$,
shorten accordingly the vertical segments whose upper endpoints lie on $r$, 
and extend $v$ upwards to $\mathrm{N}$ while truncating 
\end{minipage}
\hspace{8pt}
\begin{minipage}{0.3\textwidth}
\vspace{80pt}
\includegraphics[scale = 0.95]{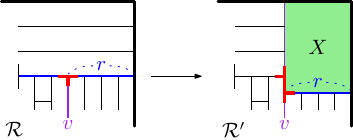}
\end{minipage}

\vspace{-2pt}
from the right the horizontal segments that it meets. 
The new rectangle $X$ is bounded by $\mathrm{N}$, $\mathrm{E}$, and (modified) $r$ and $v$.
If we perform this operation on the $i$-th from the right active $\td$ joint,
where $1 \leq i \leq \ell$,
this increases $k$ by $1$ and sets $\ell$ to be $i-1$, since
the horizontal segments of active $\td$~joints to the right of the chosen one still reach $\mathrm{E}$,
but those of active $\td$ joints to the left of the chosen one are truncated.
Therefore, the new rectangulation $\R'$ has the type $(k+1, i-1)$, and we obtain the succession rule $(**)$.
\item Inserting a new rectangle that extends from \textrm{N} to \textrm{S} at the right of $\R$.
The new rectangulation $\R'$ has the type $(k+1, 0)$. This yields the succession rule $(***)$.
\end{itemize}

\begin{figure}[h]
\begin{center}
\includegraphics[width=\textwidth]{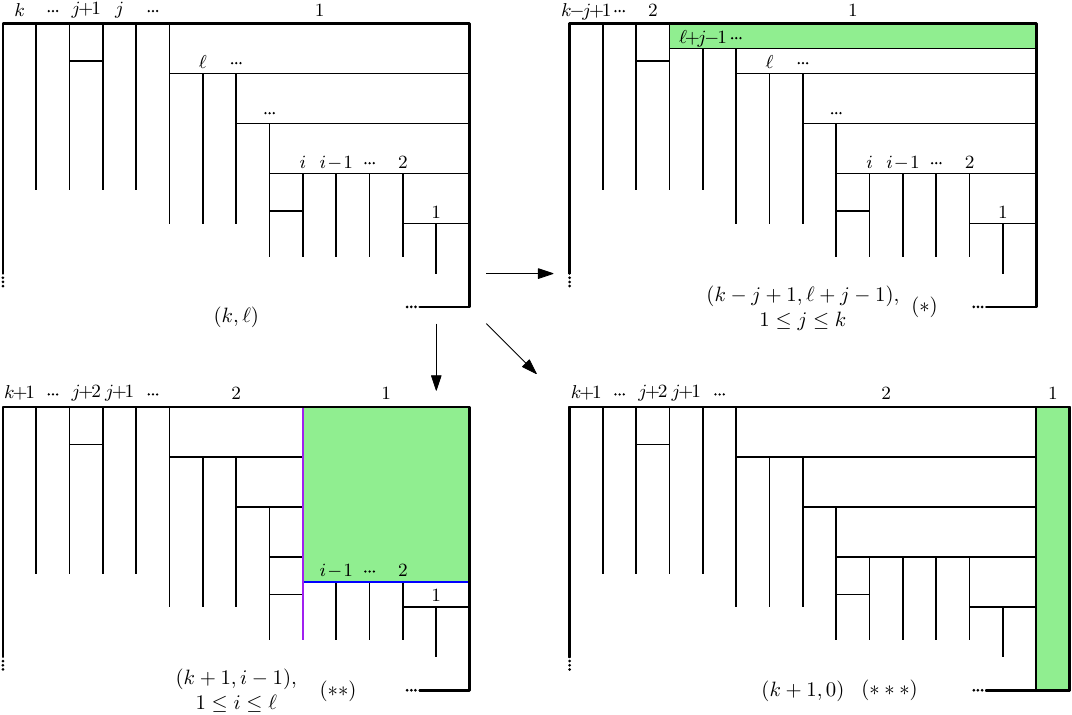}
\caption{Illustration for the succession rules of T2 for $R^s(\tu)$.}
\label{fig:gt2re}
\end{center}
\end{figure}

These succession rules ensure that every rectangulation in $R^s(\tu)$ is generated precisely once.
To see that, consider $\R' \in R^s_{n+1}(\tu)$, and let $X$ be its $\mathrm{NE}$-rectangle.
Note that the $\mathrm{SW}$-corner of $X$ has the $\tr$ shape, unless the corner lies on the boundary of $R$.
If $X$ is not an $\mathrm{S}$-rectangle and the left side of $X$ has no neighboring segment on the left, 
then $\R'$ is obtained from $\R\in R^s_{n}(\tu)$ by rule $(*)$,
whereas $\R$ is the rectangulation obtained from $\R'$ by the corner rectangle deletion.
If $X$ is not an $\mathrm{S}$-rectangle and the left side of $X$ has at least one neighboring segment on the left, 
then $\R'$ is obtained from $\R\in R^s_{n}(\tu)$ by rule $(**)$.
To obtain $\R$ 
from $\R'$ in this case, we shift the bottom side of $X$ upwards
(while respective extending of its bottom neighbors) until it aligns with the lowest left neighboring segment
and, thus, a $\crs$ shape is created; 
we cut off the portion of the vertical segment from the crossing point upwards;
and if there are some horizontal segments whose right endpoints 
were on the vertical segment which is now cut off, we extend them to the right so that they touch~$\mathrm{E}$. 
Finally, if $X$ is an $\mathrm{S}$-rectangle
then $\R'$ is obtained from $\R\in R^s_{n}(\tu)$ by rule $(***)$,
whereas $\R$ is the rectangulation obtained from $\R'$ by deleting the right-most rectangle.
Since at every step of constructing a rectangulation we insert a new $\mathrm{NE}$-rectangle,
we refer to it as \textit{constructing by $\mathrm{NE}$ insertion}
(although it is different from the ``classical'' insertion of a $\mathrm{NE}$-rectangle when $(**)$ is applied).

\medskip

Thus, T2 is a generating tree for $R^s(\tu)$.
Since it is also a generating tree for $I(011, 201)$, the theorem follows.
\end{proof}

Let $\sigma\colon R^s(\tu) \to I(011, 201)$ be the bijection that respects the succession rules of T2 for these two classes.
Next we prove two propositions that eventually lead to an explicit description of $\sigma$.

\smallskip

Let $e \in I(011, 201)$, and let $M$ be the maximum value of $e$. 
Since $011$ is forbidden, all the non-$0$ admissible values are not values of $e$.
As mentioned above, the values admissible for $(*)$ are precisely the numbers from $M+1$ to $n$, 
and the unique value admissible for $(***)$ is $0$. 
For the values admissible for $(**)$ we know that they are greater than $0$ (which is the lowest right-to-left minimum of $e$)
and smaller than $e_n$ (which is the highest right-to-left minimum of $e$).
In order to give an explicit description of $\sigma$, we need a more precise
information concerning the position of values admissible for $(**)$ between the right-to-left minima of~$e$.

Let $(0=) \ e_{b_0}<e_{b_1}<\ldots<e_{b_u} (= e_n)$ be the right-to-left minima of $e$.
Denote by $\ell_i$, $1 \leq i \leq u$, the number of values admissible for $(**)$ that lie between $e_{b_{i-1}}$ and $e_{b_{i}}$.
Then we have $\ell_1+\ell_2+\ldots+\ell_{u} = \ell$, where $\ell$ is as in the definition of T2.
Then we can provide a precise description of the values admissible for $(**)$.
\begin{prop}\label{thm:admissible2}
For every $1 \leq i \leq u$, the values admissible for $(**)$ that lie between $e_{b_{i-1}}$ and $e_{b_{i}}$
are precisely $e_{b_{i}}-\ell_i, e_{b_{i}}-\ell_i+1, \ldots, e_{b_{i}}-1$.
In other words, these are precisely the $\ell_i$ values just below $e_{b_{i}}$.
Moreover, $e_{b_{i}}-\ell_i-1$ is a value already used in $e$, 
and hence, the values admissible for $(**)$ are precisely the values just below $e_{b_{i}}$,
which are not used in $e$.
\end{prop}
\begin{proof}
As mentioned above, every value admissible for $(**)$ is not used in $e$, since the pattern $011$ is avoided.
All we need to prove is that it cannot happen that $e_{b_{i-1}}<a-1<a<e_{b_{i}}$ such that $a-1$ is admissible and $a$ is not admissible.
Assume for contradiction that this happens.
If $a$ is a value of $e$, then we have $a=e_j$ for some $j < b_{i-1}$, since $e_{b_{i-1}}$ and $e_{b_{i}}$ are two adjacent right-to-left minima.
However, in this case $(a)(e_{b_{i-1}})(a-1)$ is an~occurrence of $201$, contradicting the assumption that $a-1$ is admissible.
If $a$ is not a value of $e$, then it forms an occurrence of $201$ with some earlier two elements of $e$.
However, in this case $a-1$ forms an occurrence of $201$ with the same two elements.

Similarly, if $e_{b_{i}}-\ell_i-1$ not used in $e$, then, since it is not admissible,
it forms an occurrence of $201$ with some earlier two elements.
However, then $e_{b_{i}}-\ell_i$ forms an occurrence of $201$ with the same two elements.
\end{proof}

Let $\R \in R^s_n(\tu)$. Label its rectangles by $X_1, X_2, \ldots, X_n$ according to the order 
in which they were inserted in the $\mathrm{NE}$-insertion.

\begin{prop}\label{thm:E}
Let $e = \sigma(\R) \in I_n(011, 201)$.
Rectangle $X_j$ of $\R$ is an $\mathrm{E}$-rectangle is and only if $e_j$ is a right-to-left minimum of $e$.
Moreover, if we denote such rectangles by $Y_0, Y_1, \ldots, Y_u$,
then the number of active $\td$-joints that lie on the bottom side of $Y_{\alpha}$
is equal to the number of values admissible for $(**)$ just below $e_{b_\alpha}$.
\end{prop}
\begin{proof}
We proceed by induction. 
Suppose that some succession rule of T2 produces $\R'\in R^s_{n+1}(\tu)$ from $\R \in R^s_{n}(\tu) $, 
and $e'  \in I_{n+1}(011, 201) $ from $e \in I_{n}(011, 201)$,
and the claim holds for $\R$ and $e$. 

\textbf{Case 1.}  If $\R'$ is obtained from $\R$ and $e'$ from $e$ by $(*)$, 
specifically by $(k,\ell) \to (k-j+1, \ell+j-1)$,
then $\R'$ has a new $\mathrm{E}$-rectangle $X_{n+1}$,
whose bottom side has $j-1$ $\td$ joints.
On the other hand, $e'$ has a new right-to-left minimum $e_{n+1} = M+j$ 
which yields $j-1$ values admissible for $(*)$, 
namely $M+j-1, M+j-2, \ldots, M+1$.
There~are no changes for existing $\mathrm{E}$-rectangles and their $\td$-joints,
and for existing right-to-left minima of $e'$ and values admissible for $(*)$.

\begin{figure}[htbp]
\begin{center}
\includegraphics[scale=1]{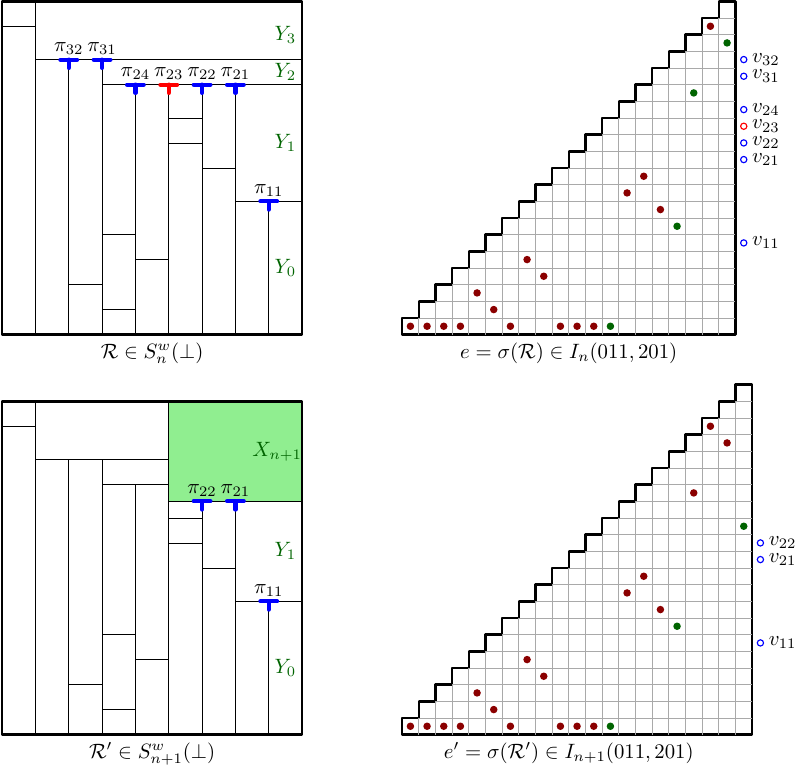} 
\end{center}
\caption{Illustration to the proof of Proposition~\ref{thm:E}, Case~2.}
\label{fig:sigmaproof}
\end{figure}

\textbf{Case 2.} Suppose $\R'$ is obtained from $\R$ and $e'$ from $e$ by $(**)$.
In $e$, let $v_{\alpha \beta}$ be
the $\beta$-th, bottom to top, value admissible for $(**)$ between $e_{\alpha-1}$ and $e_{\alpha}$,
where $\alpha = 1, 2, \ldots, u$ and $\beta = 1, \ldots, \ell_{\alpha}$. 
(In fact we know from Proposition~\ref{thm:admissible2} that $v_{\alpha \beta} = e_{b_\alpha} - 1 - (\ell_\alpha - \beta)$.)
In $\R$, let $\pi_{\alpha \beta}$ be
the $\beta$-th, right to left, active $\td$ joint on the bottom side of $Y_\alpha$,
where $\alpha = 1, 2, \ldots, u$ and $\beta = 1, \ldots, \ell_{\alpha}$.
(These parameters match for $\R$ and $e$ by the assumption.)

The active $\td$ joints are ordered so that
$\pi_{\alpha'\beta'}$ lies to the left of $\pi_{\alpha\beta}$
if and only if either $\alpha<\alpha'$ or $\alpha=\alpha'$ and $\beta<\beta'$.
From comparing the rule $(**)$ for $I(011, 201)$ and for $R^s(\tu)$ we know 
that the right-to-left ordering of active $\td$ joints in $\R$
corresponds to the bottom-to-top ordering of the values admissible for $(**)$ in $e$.
Therefore,~applying the rule $(**)$ such that $\pi_{\alpha\beta}$ is modified in $\R$
corresponds to adding the admissible value $v_{\alpha\beta}$ in $e$.
Then,~in $\R'$, 
the $\mathrm{E}$-rectangles are $Y_0, Y_1, \ldots, Y_{\alpha-1}, X_{n+1}$,
the active $\td$ joints of $Y_1, Y_2, \ldots, Y_{\alpha-1}$ are unchanged, 
and the active $\td$ joints of $X_{n+1}$ (which assumes the role of $Y_\alpha$ in $\R'$)
are $\pi_{\alpha1}, \pi_{\alpha2}, \ldots, \pi_{\alpha(\beta-1)}$.
And in $e'$
the~right-to-left minima are $e_{b_0}, e_{b_1}, \ldots, e_{b_{\alpha-1}}, v_{\alpha\beta}$,
the admissible for $(**)$ values just below $e_{b_1}, e_{b_2}, \ldots, e_{b_{\alpha-1}}$ are unchanged, 
and the admissible for $(**)$ values just below~$e_{\alpha}$
are $v_{\alpha1}, v_{\alpha2}, \ldots, v_{\alpha(\beta-1)}$.
See Figure~\ref{fig:sigmaproof} for an~illustration.

\textbf{Case 3.} Finally, if $\R'$ is obtained from $\R$ and $e'$ from $e$ by $(***)$, 
then $X_{n+1}$ is a unique $\mathrm{E}$-rectangle, and there are no active $\td$ joints on its bottom side;
and $e_{n+1}=0$ is a unique right-to-left minimum of $e$, and there are no admissible values below it.
\end{proof}

Next we define a labeling of the rectangles of $\R \in R^s(\tu)$ which will be used
for assigning the values to $e$ in the explicit definition of $\sigma$.
We begin with a partial order $\ph$ on the horizontal segments of $\R$.
For distinct horizontal segments $s$ and $t$, we first set $s \prec t$
if some point of $t$ is directly above some point of~$s$.
(This includes the possibility that either of these points is an endpoint of the respective segment.)
We define $\ph$ as the partial order obtained by taking the transitive closure of $\prec$. 
See Figure~\ref{fig:orders}(a), where the Hasse diagram of $\ph$ is shown by green lines.

\smallskip

\textit{Remark:} The partial order $\ph$ is an instance of \textit{heap of pieces} order, 
where the horizontal segments are seen as ``pieces'' with endpoints on fixed vertical rods.
The pieces are free to move vertically along the rods, as long as they do not touch
each other, also in their endpoints. Then the heap of pieces order
determines which of them are forced to be higher than others.
See~\cite{Krattenthaler2006, Viennot} for more information about heaps of pieces.

\smallskip

Next, we consider $\pl$, the linear extension of $\ph$ determined  by the following rule:
If $s$ and $t$ are independent in $\ph$, then we set $s \pl t$ if
$s$ lies to the left of $t$.
Equivalently, $\pl$ is obtained from $\ph$ iteratively by taking (and deleting) the left-most leaf
at every step.
(From the perspective of planar posets, $\pl$ is the \textit{left-most} linear extension of $\ph$.)
Every strong rectangulation $\R$ can be drawn so that for all such pairs $t$ lies higher than~$s$;
such a drawing will be referred to as a \textit{canonical drawing} of $\R$.
In the proof of Theorem~\ref{thm:strong2} we specified that when $(*)$ is applied, 
all the existing horizontal segments of $\R$ are lower than the bottom side of the new $\mathrm{NE}$-rectangle $X$;
and when $(**)$ is applied, all the horizontal segments that were lower than $s$ are lower than the bottom side of~$X$: a rectangulation obtained in this way is always canonical.
\begin{figure}[h]
\begin{center}
\includegraphics[scale=0.72]{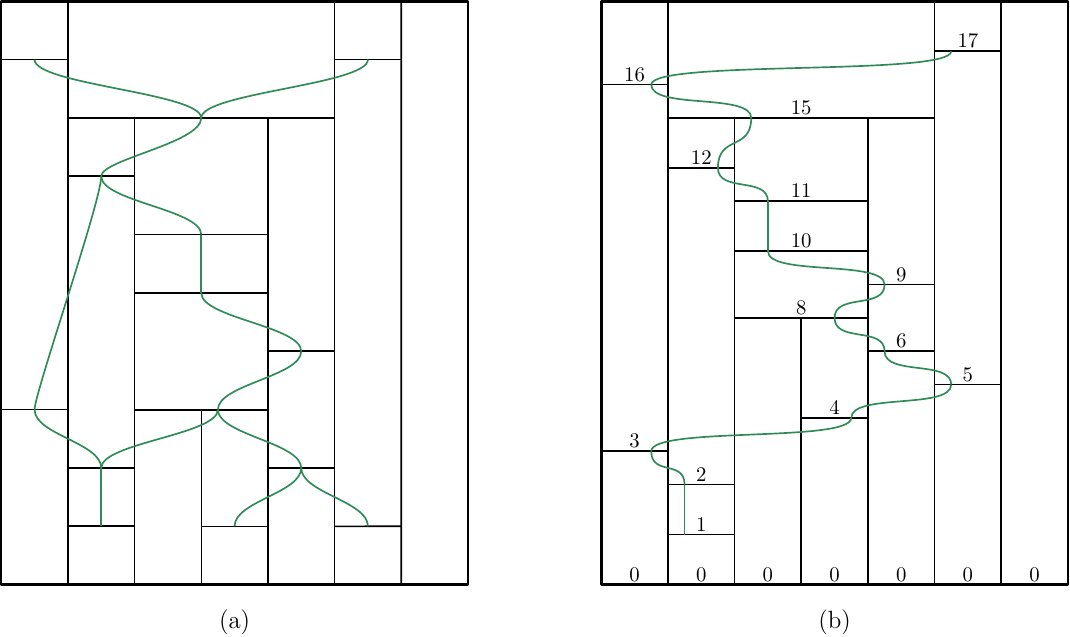} 
\caption{(a) The partial order $\ph$. 
(b) The linear order $\pl$, and the $\lambda$-labeling.}
\label{fig:orders}
\end{center}
\end{figure}

Given $\R \in R^s(\tu)$, we consider its canonical drawing,
and we label every horizontal segment $s$ by the number of rectangles of $\R$
that lie lower than $s$. This label will be denoted by $\lambda(s)$.
It is clear that $t$ lies higher than $s$ if and only if $\lambda(t) > \lambda(s)$.
Additionally, we define the $\lambda$-label of every edge of the bottom side of $\R$ to be $0$.
Note that if $t$ is the direct successor of~$s$ in $\pl$, then 
$\lambda(t)-\lambda(s)$ is the number of rectangles whose top edge is included in~$t$.
If $t$ is the lowest segment of $\R$, then 
$\lambda(t)$ is just the number of rectangles whose top edge is included in $t$.

Finally, we also define $\lambda$-labels for rectangles of $\R$.
Note that if a rectangle does not touch $\mathrm{S}$, 
then its lower edge is an entire segment.
The $\lambda$-label of a rectangle $X$ is defined to be identical to the $\lambda$-label of 
its lower edge, and denoted by $\lambda(X)$. 
In particular, we have $\lambda(X) = 0$ if and only if $X$ is an $\mathrm{S}$-rectangle.
See Figure~\ref{fig:strong2}(b): every label is the $\lambda$-label of the respective rectangle, 
and also of its lower edge.

\begin{theorem}\label{thm:sigma_explicit}
Let $\R \in R^s_n(\tu)$. 
Then $e = \sigma(\R)\in I_{n}(011, 201)$ is obtained by taking the $\lambda$-labels
of the rectangles of $\R$ in the order of their insertion 
by the succession rules of T2.
\end{theorem}
\begin{proof}
We proceed by induction.
Assume that $\R \in R^s_{n}(\tu)$ and $e = \sigma(\R) \in I_{n}(011, 201)$, 
$\R' \in R^s_{n+1}(\tu) $ is obtained from $\R$ by one of the succession rules of T2,
and $e' \in I_{n+1}(011, 201)$ is obtained by the same succession rule from~$e$.
We assume that both $\R$ and $\R'$ are given as canonical drawings.

Suppose that $\R'$ is obtained from $\R$ and $e'$ from $e$ by $(*)$,
specifically by $(k,\ell) \to (k-j+1, \ell+j-1)$.
Then we have $\lambda(X_{n+1}) = (n+1) - (k-j)$ because the $k-j$ left-most $\mathrm{N}$-rectangles are 
the only rectangles of $\R'$ which are not lower than $X_n$.
On the other hand, the final value of $e$ is in this case $e_{n+1} = M+j = n+1-k+j$.
Therefore, the $\lambda$-label of the last inserted rectangle is equal to $e_{n+1}$,
and hence, we have $\sigma(\R')=e'$.

Suppose that $\R'$ is obtained from $\R$, and $e'$ from $e$ by $(**)$,
specifically $\R'$ by modifying the $\td$ joint $\pi_{\alpha\beta}$
and $e'$ by adding $e_{n+1}=v_{\alpha\beta} =  e_{b_\alpha} - 1 - (\ell_\alpha - \beta)$.
Then $\lambda(X_{n+1})$ is equal to $e_{n+1}$,
since in $\R$ the $\lambda$-label of the modified horizontal segment was 
$e_{b_\alpha}$ by induction, and in $\R'$ just $e_{b_\alpha} - (\ell_\alpha +1 - \beta)$ of these rectangles 
are lower than $X_{n+1}$.

Finally, if $\R'$ is obtained from $\R$ and $e'$ from $e$ by $(***)$,
then the $\lambda$-label of the right-most rectangle is $0$,
and also $e_{n+1}=0$.
\end{proof}

To summarize, given $\R \in R^s(\tu)$, we proceed as follows:
We consider its canonical drawing and use it to determine the $\lambda$-labeling:
the $\lambda$-labels of the bottom edges are $0$, and then we assign the labels going
from bottom to top, and the $\lambda$-label of every segment $s$ 
is the $\lambda$-label of the previous segment plus the number of rectangles just below $s$.
To determine the order of taking rectangles, we \textit{delete} $\mathrm{NE}$ rectangles
and then reverse the order. 

\vspace{4pt}

\nid\begin{minipage}{0.65\textwidth}\setlength{\parindent}{1.5em}
There is an even more explicit description of $\sigma$.
Consider the directed graph $T$ whose nodes are rectangles of $\R$, 
and there is an edge from $X$ to $Y$ if 
the \textrm{SE}-corner of $X$ lies on the left side of $Y$. 
If there is just one $\mathrm{E}$-rectangle $K$ 
(in this case we refer to $K$ as the \textit{right-most rectangle}), then 
$T$ is a (reversed) planar rooted tree with root $K$.
If there is no such $K$, we augment $\R$ by such a rectangle,
which corresponds to adding a $0$ value to $e$.
\end{minipage}
\hspace{8pt}
\begin{minipage}{0.3\textwidth}
\includegraphics[scale=0.7]{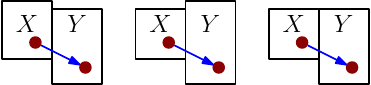}
\end{minipage}

\vspace{2pt}
It is easy to see by considering the succession rules that the last 
inserted $\mathrm{NE}$-rectangle is always the last node of $T$ 
(excluding the auxiliary rectangle $K$) in the \textit{reverse post-order} (that is, for every node, 
we recursively traverse its subtrees from right to left, 
and then the node itself).\footnote{We use the convention that
the right-to-left ordering of the subtrees of a node refers to the drawing in which 
the root is drawn at the top and all the edges are directed upwards.
For example, for the rectangle labeled $5$ in Figure~\ref{fig:strong2}(a),
the right-most subtree contains rectangles $8,6$,
the next subtree contains rectangles $10,12,11,9$,
and the left-most subtree contains rectangles $16,15$.}
Hence, we can  describe $\sigma$ as follows:
this is an inversion sequence obtained by writing the $\lambda$-labels of 
the rectangles of $\R$, when they are taken in the reverse post-order of $T$.
See Figure~\ref{fig:strong2}(a), where the tree $T$ is shown in blue.

\begin{figure}[h]
\begin{center}
\includegraphics[width=0.9\textwidth]{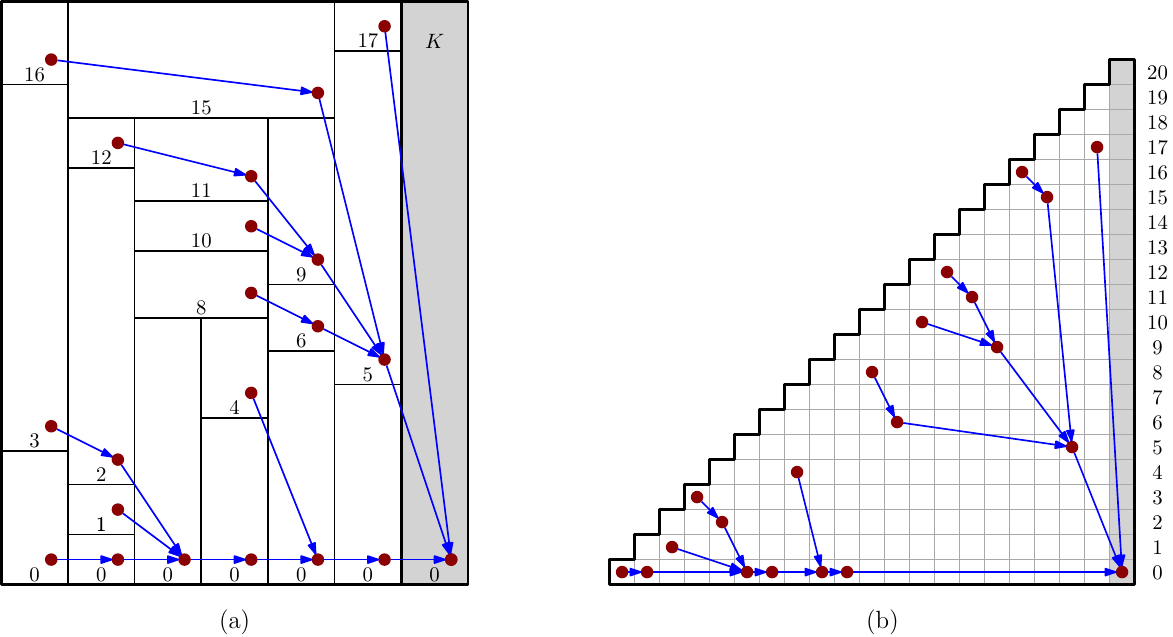} 
\end{center}
\caption{Illustration to the proof of Theorem~\ref{thm:strong2}:
(a) A canonical drawing of~$\R$ and the tree $T$ (the root $K$ is highlighted by grey).
(b) The inversion sequence $f = \sigma(\R)\in I(011, 201)$.}
\label{fig:strong2}
\end{figure}

Finally, we can demonstrate the structural identity of $R^s(\tu)$ and $I(011, 201)$
by drawing the corresponding tree over the plot of $e = \sigma(\R)$.
For two elements $e_i$ and $e_j$ of $e \in I(011, 201)$,
we say that $(e_i, e_j)$ is an \textit{inversion} if $i<j$ and $e(j) \leq e(i)$.
An inversion $(e_i, e_j)$ is \textit{minimal}
if there is no $\ell$ such that $i<\ell< j$ and $e(j)\leq e(\ell)<e(i)$.
It is easy to see that for every $e_i \neq 0$ except the last element,
(assuming that the last value of $e$ is $0$)
there is a unique $j$ such that $(e_i, e_j)$ is a minimal inversion.
Then $X_j$ is a successor of $X_i$ if and only if $(e_i, e_j)$ is an~inversion;
moreover,  $X_j$ is a successor of $X_i$ if and only if $(e_i, e_j)$ is a minimal inversion:
both statements are easily seen inductively by following the succession rules.
It~follows that if we draw the tree over the plot of $e$
where the nodes are the (dots representing the) elements of $e$
and the directed edges connect minimal inversions,
then it is isomorphic the to the embedding of $T$ in the drawing of $\R$.

\medskip

The following statistics of $I(011, 201)$ match the numbers of rectangles that touch
the sides of $R$.

\begin{prop}\label{thm:stat2}
Let $\R \in R^s(\tu)$, and let $e = \sigma (\R)$. Then: 
\vspace{-5pt}
\begin{itemize}
\item[(a)] The number of rectangles of $\R$ that touch $\mathrm{N}$ is the bounce of $e$.
\item[(b)] The number of rectangles of $\R$ that touch $\mathrm{E}$ is the number of right-to-left minima of $e$.
\item[(c)] The number of rectangles of $\R$ that touch $\mathrm{S}$ is the number of $0$ elements in $e$.
\item[(d)] The number of rectangles of $\R$ that touch $\mathrm{W}$ is the number of high elements in $e$.
\end{itemize}
\end{prop} 
\begin{proof}
(a) The number of $\mathrm{N}$-rectangles is $n$ minus the $\lambda$-label of the highest segment,
hence $n-M$.
(b)~This~follows from Proposition~\ref{thm:E}.
(c) A new rectangle is an $\mathrm{S}$-rectangle if and only if the rule $(***)$ is applied. The same rule for $e$ adds $0$.
(d) A new rectangle is a $\mathrm{W}$-rectangle if and only if the rule $(*)$ with $j=k$ is applied. The same rule adds a high element to $e$.
\end{proof}

\subsubsection{Summary concerning $I(010,101,120,201)$ and $I(011,201)$}
\label{sec:conj}
Our bijections $\tau^{(7)}$ and $\sigma$ together provide the proof of the conjecture by Yan and Lin.
Moreover, their explicit forms lead to several matching statistics in these classes.

\begin{theorem}\label{thm:conj}
For every $n \geq 1$:
\begin{itemize}
\item[(a)] We have $|I_n(010,101,120,201)| = |I_n(011,201)|$.
\item[(b)] The quadruple of statistics $(a,b,c,d)$ 
for the classes $I_n(010,101,120,201)$, $I_n(010,110,120,210)$, and $I_n(010,100,120,210)$,  where
\vspace{-3.7pt}
\begin{itemize}
\item[] $a$ is the number of $0$ elements,
\item[] $b$ is the number of left-to-right-maxima,
\item[] $c$ is the bounce,
\item[] $d$ is the number of high elements.
\end{itemize}
\vspace{-3.7pt}
matches the quadruple of statistics $(x,y,z,t)$ for the class $I_n(011,201)$, where
\vspace{-3.7pt}
\begin{itemize}
\item[] $x$ is the number of high elements, 
\item[] $y$ is the number of $0$ elements,
\item[] $z$ is the number of right-left-minima,
\item[] $t$ is the bounce. 
\end{itemize}
\vspace{-3.7pt}
Moreover, $a$ can be switched with $c$, and $x$ with $z$.
\end{itemize}
\end{theorem}
\begin{proof}
\begin{itemize}
\item[(a)] This follows from Theorems~\ref{thm:strong1} and~\ref{thm:strong2}
where we proved, by means of bijections $\tau^{(7)}$ and $\sigma$,
that $\td$-avoiding strong rectangulations are equinumerous
to both of these classes of inversion sequences.
\item[(b)] 
This follows from Propositions~\ref{thm:stat1} and~\ref{thm:stat2}
(recall that we used $\td$-avoiding rectangulations for $\tau^{(7)}$, 
and $\tu$-avoiding rectangulations for $\sigma$).
Moreover, due to the horizontal symmetry of these classes
(that is, $\R$ belongs to $R^s(\td)$
if and only if 
the horizontal reflection of $\R$ belongs to $R^s(\td)$)
$a$ can be switched with~$c$,
and, similarly, $x$ with $z$. \hfill \qedsymbol
\end{itemize}
\renewcommand{\qed}{} 
\end{proof}

\subsection{Enumeration and bijections for $R(\td,\tu)$}
\label{sec:td_tu}

\noindent
\begin{minipage}{0.76\textwidth}
\setlength{\parindent}{1.5em}
In this section we consider $(\td,\tu)$-avoiding rectangulations.
These are precisely the rectangulations in which every vertical segment reaches $\mathrm{N}$ and $\mathrm{S}$.
They can be obtained in two steps:
The empty rectangle $R$ is first partitioned by vertical cuts into $k \geq 1$ ``primary'' rectangles  
$A_1, A_2, \ldots, A_{k}$,
and then every rectangle $A_i$ is further partitioned by horizontal segments into $\ell_i \geq 1$ ``secondary'' rectangles. 
We will show that enumeration of $(\td,\tu)$-avoiding weak rectangulations is elementary, 
but $(\td,\tu)$-avoiding strong rectangulations are bijective to an interesting family of 
Dyck paths --- the so called \textit{rushed Dyck paths}.
\end{minipage}
\hspace{15pt}
\begin{minipage}{0.2\textwidth}
\includegraphics[scale=0.8]{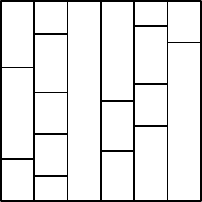} 
\end{minipage}

\vspace{15pt}

\subsubsection{$(\td,\tu)$-avoiding weak rectangulations: bijection with compositions}\label{sec:vdash_dashv_weak}
\begin{prop}\label{prop:vdash_dashv_weak}
For every $n \geq 1$ we have $|R^w_n(\td,\tu)| = 2^{n-1}$.
\end{prop}
\begin{proof} For every $n$ there is a bijection between 
$(\td,\tu)$-avoiding rectangulations of size $n$
and compositions of~$n$ which maps $\R$ to the composition $n=\ell_1+\ell_2 + \ldots + \ell_k$,
where $\ell_1, \ell_2,  \ldots, \ell_k$ are as above.
For example, the~rectangulation in
the drawing corresponds to the composition $18=3+5+1+3+4+2$. 
\end{proof}

\subsubsection{$(\td,\tu)$-avoiding strong rectangulations: bijection with rushed Dyck paths}
\label{sec:td_tu_strong}
\label{RUSH}

We use the equivalent set of patterns $\{\tr,\tl\}$.
We show that $(\tr,\tl)$-avoiding strong rectangulations
are equinumerous with two families of Dyck paths: 
\textit{rushed Dyck paths} and \textit{progressive Dyck paths}.
A Dyck path is \textit{rushed} if it starts with $h \geq 1$ up-steps 
and then never visits the altitude $h$ again.
A Dyck path is \textit{progressive} if every peak at altitude $h > 1$ 
is preceded by at least one peak at the altitude $h-1$.

Both rushed Dyck paths and progressive Dyck paths are enumerated by~\href{https://oeis.org/A287709}{A287709}. 
In 2018, the first author proved, via telescoping in generating functions, that these families are equinumerous, 
and Jel\'inek~\cite{Jelinek} gave a~``semi-bijective'' proof based on the involution principle.
A direct bijective proof, along with an enumerative formula as well as asymptotic approximations, 
was recently given by Bacher~\cite{Bacher2024},
who also coined the notions of \textit{rushed} and \textit{progressive} Dyck paths. 
In the following theorem we prove that these kinds of Dyck paths 
are equinumerious with $(\tr,\tl)$-avoiding strong rectangulations.

\begin{theorem}\label{thm:rush}
For every $n\geq 1$ and $k \geq 1$,
the number of rushed Dyck paths of semilength $n+1$
and height $k+1$
is equal 
to the number of $(\tr,\tl)$-avoiding strong rectangulations of size $n$
with $k-1$ horizontal segments.
Therefore, $(\tr,\tl)$-avoiding strong rectangulations are enumerated by~\href{https://oeis.org/A287709}{A287709}. 
\end{theorem}

\begin{proof} Consider the following bijection $\varphi$ from rushed Dyck paths of semilength $\geq 2$ 
to $R^s(\tr,\tl)$, which is adapted from the canonical mapping between Dyck paths
and heaps of dimers (see for example~\cite[Lemma~2]{CiglerKrattenthaler2024} or~\cite{Viennot2021}).

Let $P$ be a rushed Dyck path of semilength $n+1$
that starts with precisely $h=k+1$ up-steps.
Consider the usual drawing of~$P$ as a lattice path from $(0,0)$ to $(2(n+1),0)$. 
Next, consider the grid rectangle $R=[h+1, \, 2(n+1)]\times[0,h-1]$, 
and draw horizontal segments along the grid lines to partition $R$ into $h-1$ rectangles
on top of each other.
For every up-step~$\mathsf{U}$ of~$P$ within $R$ 
(that is, every up-step of~$P$ which is not one of $h$ initial up-steps),
draw a vertical segment of height~$1$
whose midpoint coincides with the midpoint of $\mathsf{U}$.
Then~$\R := \varphi(P)$ is the rectangle $R$ with horizontal and vertical segments added as described.
It has $h-2 = k-1$ horizontal segments and $n+1-h$ vertical segments and hence, it is of size~$n$;
and it is clearly $(\tr,\tl)$-avoiding. 
See Figure~\ref{fig:rush} for an example. 

\begin{figure}[h]
\begin{center}
\includegraphics[width=\textwidth]{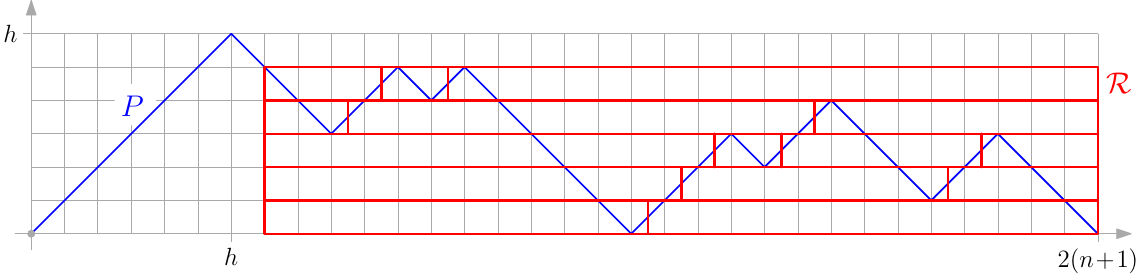} 
\end{center}
\caption{Illustration for the mapping $\varphi$ in the proof of Theorem~\ref{thm:rush}:
a rushed Dyck path $P$ and the \\ $(\tr,\tl)$-avoiding strong rectangulation $\R=\varphi(P)$.}
\label{fig:rush}
\end{figure}

The mapping $\varphi$ is a bijection, and its inverse $\varphi^{-1}$ can be described as follows:
Given a $(\tr,\tl)$-avoiding strong rectangulation $\mathcal{R}$, 
we first consider the (left-to-right) heap of pieces partial order on the vertical segments of $\R$
(see Figure~\ref{fig:rush_inv}(a)),
and then its left-most linear extension 
(see Figure~\ref{fig:rush_inv}(b)).
We draw $\R$ so that its vertical segments occur
according to this linear ordering (see Figure~\ref{fig:rush_inv}(b)).
Finally, we draw the (unique) Dyck path~$P$
such that its baseline contains the bottom side of $\mathcal{R}$,
it starts with $k+1$ up-steps, where $k$ is the number of primary rectangles in $\mathcal{R}$,
and its non-initial up-steps intersect the vertical segments according to the introduced order.
See Figure~\ref{fig:rush_inv}(c) for an illustration.

The path $P$ is clearly a rushed Dyck path. 
It satisfies $\varphi(P)=\R$ since the order of up-steps in a Dyck path
is always the left-most linear extension of their heap of pieces order.
It is unique since a Dyck path is uniquely defined by 
the sequence of the levels of its up-steps.\end{proof}

\begin{figure}[h]
\begin{center}
\includegraphics[width=\textwidth]{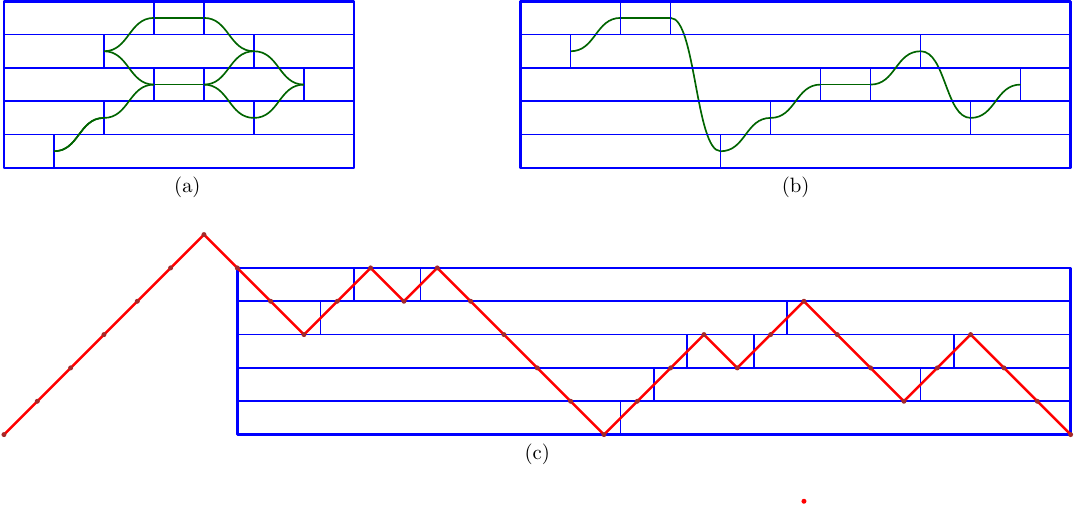} 
\end{center}
\caption{Illustration for the mapping $\varphi^{-1}$ in the proof of Theorem~\ref{thm:rush}:
(a) A $(\tr,\tl)$-avoiding strong rectangulation $\R$
and the heap of pieces poset of its vertical segments;
(b) The left-most linear extension of the poset;
(c) The rushed Dyck path $P = \varphi^{-1}(\R)$.}
\label{fig:rush_inv}
\end{figure}

The initial up-steps and the first down-step are not essential in a rushed Dyck path.
If we delete these steps and flip the path horizontally, we obtain 
an equivalent family of lattice paths with steps $(1,1)$ and $(1,-1)$
that go from $(0,0)$ to $(n,k)$
and stay (weakly) between the lines $y=0$ and $y=k$. 
This is a special case of ``Dyck paths in strips''. 
By~\cite[Example~10.11.2]{Krattenthaler2015}, the generating function for such paths for fixed~$k$ is 
$\displaystyle{\frac{1}{x \, U_{k+1}\left(\frac{1}{2x}\right)}}$,
where $U_{k+1}$ is the $(k+1)$-th Chebyshev polynomial of the second kind.
Upon the suitable change of variables, we obtain the generating function for 
$(\tr,\tl)$-avoiding strong rectangulations with $k-1$ horizontal segments:
\[g_k(x) = \displaystyle{\frac{x^\frac{k-1}{2}}{U_{k+1}\left(\frac{1}{2\sqrt{x}}\right)}}.\]
For small values of $k$, these generating functions are
$g_1(x)=\frac{x}{1-x}$, 
$g_2(x)=\frac{x^2}{1-2x}$, 
$g_3(x)=\frac{x^3}{1 - 3x + x^2}$ (\href{https://oeis.org/A001906}{A001906}, a bisection of the Fibonacci sequence),
$g_4(x)=\frac{x^4}{(1-x)(1-3x)}$ (\href{https://oeis.org/A003462}{A003462}),
$g_5(x) = \frac{x^5}{1-5x+6x^2-x^3}$ (\href{https://oeis.org/A005021}{A005021}),
$g_6(x) = \frac{x^6}{(1-2x)(1-4x+2x^2)}$ (\href{https://oeis.org/A094811}{A094811}), etc.
The asymptotic growth rate of the $k$-th sequence is $4 \cos^2 (\frac{\pi}{k+2})$.

In the work of Bacher~\cite[Theorems 4 and 5]{Bacher2024}, similar results are given
in terms of Fibonacci polynomials, and also the asymptotics 
for the number of rushed Dyck path
is derived: $\lambda \, 4^n  \, e^{-\nu n^{1/3}}  \, n^{-5/6} \, \big(1+O(n^{-1/3})\big)$,
where $\lambda \approx 4.21$ and $\nu\approx 3.18$.
Due to our result, this is also the asymptotics for $|R_n^s(\td,\tu)|$.

\subsubsection{Restricted classes of permutations and inversion sequences}

Here we provide enumerative results for classes of inversion sequences
obtained by restricting the bijections from Section~\ref{sec:td} to $\{\td, \tu\}$-avoiding rectangulations.

$\{\td, \tu\}$-avoiding rectangulations can be characterized as those 
$\td$-avoiding rectangulations in which every vertical segment is a cut. 
Accordingly, under $\tau$, $\{\td, \tu\}$-avoiding weak rectangulations 
correspond to non-decreasing inversion sequences in which all left-to-right maxima are high, 
or, equivalently, to Dyck paths in which all the valleys are at the $x$-axis.
Hence, restricting the bijection $\tau$ from Theorem~\ref{thm:td_weak}
to $\{\td, \tu\}$-avoiding rectangulations implies 
that there are $2^{n-1}$ such Dyck paths of semilength $n$ ---
however, this result is, regardless, elementary.

For strong rectangulations, restricting $\tau^{(7)}$ to $R^s(\td, \tu)$ yields the following result:

\begin{prop}\label{thm:tau_tu_td}
$(010, 101, 120, 201)$-avoiding inversion sequences in which 
all left-to-right maxima are high, 
are enumerated by \href{https://oeis.org/A287709}{A287709}.
\end{prop}

Via $\tau^{(6)}$ and $\tau^{(8)}$ we see that
$(010, 110, 120, 210)$-avoiding inversion sequences 
and $(010, 100, 120, 210)$-avoiding inversion sequences 
that satisfy the same condition
are enumerated by the same sequence.

In context of the bijection $\sigma$, 
$\{\td, \tu\}$-avoiding rectangulations 
are precisely those in which the number of $\mathrm{N}$-rectangles
is equal to the number of $\mathrm{S}$-rectangles.
Then Proposition~\ref{thm:stat2}(a,c) implies the following result.

\begin{prop}\label{thm:sigma_tu_td}
$(011, 201)$-avoiding inversion sequences $e$ in which the bounce is equal to the number of $0$-elements 
(or, equivalently: in which the set of values is precisely $\{0, 1, \ldots, M\}$, where $M$ is the maximum of $e$)
are enumerated by \href{https://oeis.org/A287709}{A287709}.
\end{prop}

\subsection{Elementary models: $R(\td,\tr)$, $R(\td,\tu,\tr)$, and $R(\td,\tu,\tr,\tl)$}
\label{sec:elementary}

The analysis of the last three classes, $R(\td,\tr)$, $R(\td,\tu,\tr)$, and $R(\td,\tu,\tr,\tl)$
is elementary to trivial. 
We~include them for the sake of completeness, 
and also indicate the family of non-decreasing inversion sequences
that correspond to these classes under $\tau$.

\vspace{7pt}

\noindent
\begin{minipage}{0.75\textwidth}
\setlength{\parindent}{1.5em}
\begin{prop}
$|R_n(\td,\tr)|=2^{n-1}$.
\end{prop}

\noindent \textit{Proof.} A rectangulation $\R$ is $(\td,\tr)$-avoiding 
if and only if every rectangle meets $\mathrm{W}$ or $\mathrm{N}$.
Indeed, if some rectangle of $\R$ meets neither $\mathrm{W}$ or $\mathrm{N}$, 
then its \textrm{NW}-corner is a part of a $\td$ or a $\tr$ joint.

Consider the \textrm{NW}--\textrm{SE} labeling of $\mathcal{R} \in R_n(\td,\tr)$.
Then $\mathcal{R}$ is uniquely determined by the binary sequence
that indicates which side of $R$ --- $\mathrm{W}$ or $\mathrm{N}$ --- the rectangles labeled $2, 3, \ldots, n$ meet. 
(For example, the rectangulation in the drawing is determined by the sequence
$\mathrm{NWWWNNWNW}$.)
Hence, there are $2^{n-1}$ such rectangulations. \hfill $\qed$
\end{minipage}
\hspace{18pt}
\begin{minipage}{0.2\textwidth}
\includegraphics[scale = 0.65]{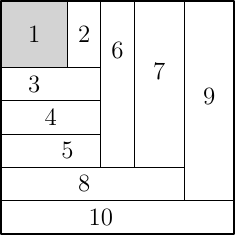}
\end{minipage}

\medskip

Under the bijection $\tau$ from Theorem~\ref{thm:td_weak}, 
$(\td,\tr)$-avoiding rectangulations correspond to non-decreasing inversion sequences in which every jump has height $1$.

\begin{obs}
$|R_n(\td,\tu,\tr)|=n.$
\end{obs}

\noindent
\begin{minipage}{0.75\textwidth}
\setlength{\parindent}{1.5em}

\begin{proof}A rectangulation in $R_n(\td,\tu,\tr)$ can be obtained by drawing $k$ vertical segments $(0\le k \le n-1)$ to obtain $k+1$ rectangles, and then adding $n-k-1$ horizontal segments in the left-most rectangle. 
Such a rectangulation is uniquely determined by the choice of $k$, hence $|R_n(\td,\tu,\tr)|=n$. 
\end{proof}

Under the bijection  $\tau$, the corresponding inversion sequences 
have the shape $(0, 0,\dots,0,n-k,n-k+1,\dots,n-1)$,
where $k$ is the number of vertical segments.
\end{minipage}
\hspace{18pt}
\begin{minipage}{0.2\textwidth}
\includegraphics[scale = 0.8]{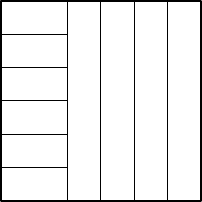}
\end{minipage}

\begin{obs}
For $n \geq 2$ we have $|R_n(\td,\tu,\tr,\tl)| = 2$.
\end{obs}
\begin{proof}
A rectangulation in $R_n(\td,\tu,\tr,\tl)$ has either all vertical or all horizontal segments. Thus, there are just two such rectangulations of size $n$.
\end{proof}

Under the bijection $\tau$, the corresponding inversion sequences 
are $(0, 0, 0, \ldots, 0)$ and $(0, 1, 2, \ldots, n-1)$.

\section{Concluding remarks}
This paper is our first report on our study of patterns in rectangulations.
In the following parts we thoroughly explore connections between rectangulation 
patterns and permutation patterns, and provide a catalog of avoidance classes determined by
\textit{small patterns} which consist of two or three segments,
which leads to interesting phenomena and new surprising links.
We hope that this work will serve as a foundation for more general 
study of pattern avoidance in rectangulations.

Returning to our proof of the conjecture by Yan and Lin, we remark that
it is was not least the visual nature of rectangulations that helped us to 
find the statistics which elucidated a link between two classes
of inversion sequences which is somewhat hidden in their structure,
and even not easily seen in the plots. This is an additional motivation for the study of rectangulations.

\hypersetup{urlcolor=blue}

\acknowledgements
\label{sec:ack}
We would like to express our thanks to Torsten Mütze, Namrata, and Aaron Williams
for fruitful conversations and for sharing their ideas and experimental results.
Our discussions took place, in particular,
during the \href{https://2023.permutationpatterns.com/}{\textit{Permutation Patterns 2023}} conference (3--7 July 2023, Dijon, France) and the research workshop \href{http://tmuetze.de/workshop24.html}{\textit{Combinatorics, Algorithms, and Geometry}} (4--8 March 2024, Dresden, Germany). We thank the organizers of both events. 

\nocite{*}
\bibliographystyle{plain}
\bibliography{AP}

\begin{thebibliography}{10}

\bibitem{AckermanBarequetPinter2006bij}
Eyal Ackerman, Gill Barequet, and Ron~Y. Pinter.
\newblock \href{https://doi.org/10.1016/j.dam.2006.03.018}{A bijection between permutations and floorplans, and its applications}.
\newblock {\em Discrete Appl. Math.}, 154(12):1674--1684, 2006.

\bibitem{AckermanBarequetPinter2006num}
Eyal Ackerman, Gill Barequet, and Ron~Y. Pinter.
\newblock \href{https://doi.org/10.1016/j.jcta.2005.10.003}{On the number of rectangulations of a planar point set}.
\newblock {\em J. Combin. Theory Ser. A}, 113(6):1072--1091, 2006.

\bibitem{AsinowskiBanderier2024}
Andrei Asinowski and Cyril Banderier.
\newblock \href{https://www.mat.univie.ac.at/~slc/wpapers/FPSAC2024/46.html}{From geometry to generating functions: {R}ectangulations and permutations}.
\newblock {\em S\'em. {L}othar. {C}ombin.}, 91B:Article 46, 1--12, 2024.
\newblock Proceedings of the 36th International Conference on Formal Power Series and Algebraic Combinatorics (FPSAC 2024).

\bibitem{AsinowskiCardinalFelsnerFusy2024}
Andrei Asinowski, Jean Cardinal, Stefan Felsner, and {\'{E}}ric Fusy.
\newblock \href{https://doi.org/10.5070/c65165025}{Combinatorics of rectangulations: {O}ld and new bijections}.
\newblock {\em Combin. Theory}, 5(1):1--57, 2025.

\bibitem{AsinowskiMansour2010}
Andrei Asinowski and Toufik Mansour.
\newblock \href{https://doi.org/10.1007/s00026-010-0043-8}{Separable $d$-permutations and guillotine partitions}.
\newblock {\em Ann. Comb.}, 14:17--43, 2010.

\bibitem{Bacher2024}
Axel Bacher.
\newblock \href{https://doi.org/10.4204/EPTCS.403.10}{Progressive and rushed {D}yck paths}.
\newblock In Sre\v{c}ko Brlek and Luca Ferrari, editors, {\em EPTCS, volume 403}, pages 29--34, 2024.
\newblock Proceedings of the 13th edition of the conference Random Generation of Combinatorial Structures. Polyominoes and Tilings (GASCom 2024).

\bibitem{Bevan2015}
David Bevan.
\newblock \href{https://arxiv.org/abs/1506.06673}{Permutation patterns:~{B}asic definitions and notation}.
\newblock {\em arXiv preprint}, arXiv:1506.06673, 2015.

\bibitem{BuchinEppsteinLoefflerNoellenburgSilveira2016}
Kevin Buchin, David Eppstein, Maarten L{\"{o}}ffler, Martin N{\"{o}}llenburg, and Rodrigo~I. Silveira.
\newblock \href{https://doi.org/10.20382/jocg.v7i1a6}{Adjacency-preserving spatial treemaps}.
\newblock {\em J. Comput. Geom.}, 7(1):100--122, 2016.

\bibitem{CallanMansour2023}
David Callan and Toufik Mansour.
\newblock \href{https://math.colgate.edu/~integers/x78/x78.pdf}{Inversion sequences avoiding quadruple length-3 patterns}.
\newblock {\em Integers}, 23:Article A78, 1--64, 2023.

\bibitem{CardinalSacristanSilveira2018}
Jean Cardinal, Vera Sacrist{\'{a}}n, and Rodrigo~I. Silveira.
\newblock \href{https://doi.org/10.23638/DMTCS-20-2-14}{A note on flips in diagonal rectangulations}.
\newblock {\em Discrete Math. Theor. Comput. Sci.}, 20(2):Article 14, 1--22, 2018.

\bibitem{CiglerKrattenthaler2024}
Johann Cigler and Christian Krattenthaler.
\newblock \href{https://doi.org/10.1016/j.ejc.2023.103840}{Bounded {D}yck paths, bounded alternating sequences, orthogonal polynomials, and reciprocity}.
\newblock {\em European J. Combin.}, 121:103840, 1--62, 2024.

\bibitem{CorteelMartinezSavageWeselcouch2016}
Sylvie Corteel, Megan~A. Martinez, Carla~D. Savage, and Michael Weselcouch.
\newblock \href{https://doi.org/10.46298/dmtcs.1323}{Patterns in inversion sequences {I}}.
\newblock {\em Discrete Math. Theor. Comput. Sci.}, 18(2):Article~2, 1--21, 2016.

\bibitem{EppsteinMumfordSpeckmannVerbeek2012}
David Eppstein, Elena Mumford, Bettina Speckmann, and Kevin Verbeek.
\newblock \href{https://doi.org/10.1137/110834032}{Area-universal and constrained rectangular layouts}.
\newblock {\em {SIAM} J. Comput.}, 41(3):537--564, 2012.

\bibitem{Felsner2013}
Stefan Felsner.
\newblock \href{https://doi.org/10.1007/978-1-4614-0110-0_12}{Rectangle and square representations of planar graphs}.
\newblock In J{\'a}nos Pach, editor, {\em Thirty essays on geometric graph theory}, pages 213--248. Springer, New York, 2013.

\bibitem{FelsnerFusyNoyOrden2011}
Stefan Felsner, {\'{E}}ric Fusy, Marc Noy, and David Orden.
\newblock \href{https://doi.org/10.1016/j.jcta.2010.03.017}{Bijections for {B}axter families and related objects}.
\newblock {\em J. Combin. Theory Ser. {A}}, 118(3):993--1020, 2011.

\bibitem{Felsner2024}
Stefan Felsner, Andrew Nathenson, and Csaba~D. T\'oth.
\newblock \href{https://doi.org/10.57717/cgt.v3i1.37}{Aspect ratio universal rectangular layouts}.
\newblock {\em Comput. Geom. Topol.}, 3(1):Article 3, 1--24, 2024.

\bibitem{Flemming1978}
Ulrich Flemming.
\newblock \href{https://doi.org/10.1068/b050215}{Wall representations of rectangular dissections and their use in automated space allocation}.
\newblock {\em Environment and Planning B: Planning and Design}, 5(2):215--232, 1978.

\bibitem{Hong2000}
Xianlong Hong, Gang Huang, Yici Cai, Jiangchun Gu, Sheqin Dong, Chung{-}Kuan Cheng, and Jun Gu.
\newblock \href{https://doi.org/10.1109/ICCAD.2000.896442}{Corner block list: {A}n effective and efficient topological representation of non-slicing floorplan}.
\newblock In Ellen Sentovich, editor, {\em Proceedings of the {IEEE/ACM} International Conference on Computer Aided Design (ICCAD 2000)}, pages 8--12. {IEEE} Computer Society, 2000.

\bibitem{Jelinek}
V\'it Jel\'inek.
\newblock Personal communication, 2018.

\bibitem{KotsireasMansourYildirim2024}
Ilias Kotsireas, Toufik Mansour, and G\"okhan Y{\i}ld{\i}r{\i}m.
\newblock \href{https://doi.org/10.1016/j.jsc.2023.102231}{An algorithmic approach based on generating trees for enumerating pattern-avoiding inversion sequences}.
\newblock {\em J. Symbolic Comput.}, 120:Article 102231, 1--18, 2024.

\bibitem{KozmaSaranurak2016}
L{\'{a}}szl{\'{o}} Kozma and Thatchaphol Saranurak.
\newblock \href{http://arxiv.org/abs/1603.08151}{Binary search trees and rectangulations}.
\newblock {\em arXiv preprint}, arXiv:1603.08151, 2016.

\bibitem{Krattenthaler2015}
Christian Krattenthaler.
\newblock Lattice path enumeration.
\newblock In {\em \href{https://doi.org/10.1201/b18255}{Handbook of enumerative combinatorics}}, Discrete Math. Appl. (Boca Raton), pages 589--678. CRC Press, Boca Raton, FL, 2015.
\newblock \textit{arXiv preprint:} \href{https://doi.org/10.48550/arXiv.1503.05930}{arXiv:1503.05930 (2015)}.

\bibitem{Krattenthaler2006}
Christian Krattenthaler.
\newblock \href{https://www.mat.univie.ac.at/~kratt/artikel/heaps.pdf}{The theory of heaps and the Cartier--Foata monoid.}
\newblock Appendix in \textit{Commutation and Rearrangements}, an electronic reedition of: Pierre Cartier, Dominique Foata, Probl\`emes combinatoires de commutation et r\'earrangements (Lecture Notes in Mathematics, Vol.~85, 1969), S\'eminaire Lotharingien de Combinatoire / Books, 2006.

\bibitem{KreveldSpeckmann2007}
Marc~van Kreveld and Bettina Speckmann.
\newblock \href{https://doi.org/10.1016/j.comgeo.2006.06.002}{On rectangular cartograms}.
\newblock {\em Comput. Geom.}, 37(3):175--187, 2007.

\bibitem{LaPotinDirector1986}
David~P. La~Potin and Stephen~W. Director.
\newblock \href{https://doi.org/10.1109/TCAD.1986.1270219}{Mason: {A} Global Floorplanning Approach for {VLSI} Design}.
\newblock {\em {IEEE} Trans. Comput. Aided Des. Integr. Circuits Syst.}, 5(4):477--489, 1986.

\bibitem{LawReading2012}
Shirley Law and Nathan Reading.
\newblock \href{https://doi.org/10.1016/j.jcta.2011.09.006}{The {H}opf algebra of diagonal rectangulations}.
\newblock {\em J. Combin. Theory Ser. {A}}, 119(3):788--824, 2012.

\bibitem{Lengauer2012}
Thomas Lengauer.
\newblock {\em \href{https://doi.org/10.1007/978-3-322-92106-2}{Combinatorial Algorithms for Integrated Circuit Layout}}.
\newblock Applicable Theory in Computer Science. Vieweg+Teubner Verlag, Wiesbaden, 1990.

\bibitem{MansourShattuck2015}
Toufik Mansour and Mark Shattuck.
\newblock \href{https://doi.org/10.1515/puma-2015-0016}{Pattern avoidance in inversion sequences}.
\newblock {\em Pure Math. Appl. (PU.M.A.)}, 25(2):157--176, 2015.

\bibitem{MartinezSavage2018}
Megan Martinez and Carla Savage.
\newblock \href{https://www.emis.de/journals/JIS/VOL21/Martinez/mart2.html}{Patterns in inversion sequences II: Inversion sequences avoiding triples of relations}.
\newblock {\em J. Integer Seq.}, 21(2):Article 18.2.2, 1--44, 2018.

\bibitem{Meehan2019bax}
Emily Meehan.
\newblock \href{https://doi.org/10.37236/7273}{Baxter posets}.
\newblock {\em Electron. J. Combin.}, 26(3):Article P3.33, 1--24, 2019.

\bibitem{Meehan2019hopf}
Emily Meehan.
\newblock \href{https://arxiv.org/abs/1903.09874}{The {H}opf algebra of generic rectangulations}.
\newblock {\em arXiv preprint}, arXiv:1903.09874, 2019.

\bibitem{MerinoMuetze2023}
Arturo Merino and Torsten M\"{u}tze.
\newblock \href{https://doi.org/10.1007/s00454-022-00393-w}{Combinatorial generation via permutation languages. {III}. {R}ectangulations}.
\newblock {\em Discrete Comput. Geom.}, 70(1):51--122, 2023.

\bibitem{MitchellSteadmanLiggett1976}
William~J. Mitchell, J.~Philip Steadman, and Robin~S. Liggett.
\newblock \href{https://doi.org/10.1068/b030037}{Synthesis and optimization of small rectangular floor plans}.
\newblock {\em Environment and Planning B: Planning and Design}, 3(1):37--70, 1976.

\bibitem{MuetzeNamrata}
Torsten M\"utze and Namrata.
\newblock Personal communication, 2023.

\bibitem{oeis}
The {O}n-{L}ine {E}ncyclopedia of {I}nteger {S}equences.
\newblock Published electronically at \href{http://oeis.org/}{https://oeis.org/}.

\bibitem{Pantone2024}
Jay Pantone.
\newblock \href{https://doi.org/10.54550/eca2024v4s4r25}{The enumeration of inversion sequences avoiding the patterns 201 and 210}.
\newblock {\em Enumer. Combin. Appl.}, 4(4):Article S2R25, 1--12, 2024.

\bibitem{Reading2012}
Nathan Reading.
\newblock \href{https://doi.org/10.1016/j.ejc.2011.11.004}{Generic rectangulations}.
\newblock {\em Eur. J. Combin.}, 33(4):610--623, 2012.

\bibitem{Richter2022}
David Richter.
\newblock \href{https://doi.org/10.55016/ojs/cdm.v17i2.72522}{Some notes on generic rectangulations}.
\newblock {\em Contrib. Discrete Math.}, 17(2):41--66, 2022.

\bibitem{Stanley2015}
Richard~P. Stanley.
\newblock {\em \href{https://doi.org/10.1017/CBO9781139871495}{Catalan Numbers}}.
\newblock Cambridge University Press, 2015.

\bibitem{Steadman1983}
John~Philip Steadman.
\newblock {\em \href{https://www.researchgate.net/publication/238786010_Architectural_morphology_an_introduction_to_the_geometry_of_building_plans}{Architectural Morphology: An Introduction to the Geometry of Building Plans}}.
\newblock Pion, 1983.

\bibitem{Testart2024}
Benjamin Testart.
\newblock \href{https://arxiv.org/abs/2407.07701}{Completing the enumeration of inversion sequences avoiding one or two patterns of length 3}.
\newblock {\em arXiv preprint}, arXiv:2407.07701, 2024.

\bibitem{Viennot}
Gérard~Xavier Viennot.
\newblock \href{https://doi.org/10.1007/BFb0072524}{Heaps of pieces, {I}: {B}asic definitions and combinatorial lemmas}.
\newblock In Gilbert Labelle and Pierre Leroux, editors, {\em Combinatoire énumérative}, volume 1234 of {\em Lecture Notes in Mathematics}, pages 321--350. Springer Berlin Heidelberg, 1986.
\newblock Proceedings of Colloque de Combinatoire Enumerative, Montreal, Canada, 1985.

\bibitem{Viennot2021}
Gérard~Xavier Viennot.
\newblock \href{https://viennot.org/wa_res/files/luminy21-web2.pdf}{Lattice Paths and Heaps}.
\newblock Lecture given at the Conference on Lattice Paths, Combinatorics and Interactions, CIRM, Luminy, 2021. \href{https://viennot.org/abjc-cirm21.html}{Video} available at \href{https://viennot.org}{https://viennot.org}.

\bibitem{Viennot2}
Gérard~Xavier Viennot.
\newblock \href{https://www.viennot.org/abjc1.html}{The Art of Bijective Combinatorics, Part I: An Introduction to Enumerative, Algebraic and Bijective Combinatorics.}
\newblock IMSc, Chennai, 2016. \href{https://www.viennot.org/abjc1-contents.html}{Video} available at \href{https://viennot.org}{https://viennot.org}.

\bibitem{West1996}
Julian West.
\newblock \href{https://doi.org/10.1016/S0012-365X(96)83023-8}{Generating trees and forbidden subsequences}.
\newblock {\em Discrete Math.}, 157(1-3):363--374, 1996.

\bibitem{Williams}
Aaron Williams.
\newblock Personal communication, 2023.

\bibitem{WimerKorenCederbaum1988}
Shmuel Wimer, Israel Koren, and Israel Cederbaum.
\newblock \href{https://doi.org/10.1109/31.1739}{Floorplans, planar graphs, and layouts}.
\newblock {\em IEEE Trans. Circuits and Systems}, 35(3):267--278, 1988.

\bibitem{YanLin2020}
Chunyan Yan and Zhicong Lin.
\newblock \href{https://doi.org/10.23638/DMTCS-22-1-23}{Inversion sequences avoiding pairs of patterns}.
\newblock {\em Discrete Math. Theor. Comput. Sci.}, 22(1):Article 23, 1--35, 2020.

\bibitem{YaoChenChengGraham2003}
Bo~Yao, Hongyu Chen, Chung{-}Kuan Cheng, and Ronald~L. Graham.
\newblock \href{https://doi.org/10.1145/606603.606607}{Floorplan representations: {C}omplexity and connections}.
\newblock {\em {ACM} Trans. Design Autom. Electron. Syst.}, 8(1):55--80, 2003.

\end{thebibliography}
\label{sec:biblio}

\end{document}